\documentclass{article}

\PassOptionsToPackage{numbers}{natbib} % numerical citations regardless of build order
\usepackage[main, final]{neurips_2026}

\usepackage[utf8]{inputenc} % allow utf-8 input
\usepackage[T1]{fontenc}    % use 8-bit T1 fonts
\usepackage{hyperref}       % hyperlinks
\usepackage{url}            % simple URL typesetting
\usepackage{booktabs}       % professional-quality tables
\usepackage{amsfonts}       % blackboard math symbols
\usepackage{nicefrac}       % compact symbols for 1/2, etc.
\usepackage{microtype}      % microtypography
\usepackage{xcolor}         % colors
\usepackage{amsmath,amssymb,amsfonts}
\usepackage{amsthm}
\usepackage{algorithm}
\usepackage{algpseudocode}
\usepackage{wrapfig}
\usepackage{enumitem}
\usepackage{bm}
\usepackage{scalerel}
\usepackage{graphicx}
\usepackage{float}
\usepackage{array}
\usepackage{colortbl}
\usepackage{tikz}
\usepackage{subcaption}

\newtheorem{theorem}{Theorem}[section]
\newtheorem{lemma}[theorem]{Lemma}

\newtheorem{definition}[theorem]{Definition}
\newtheorem{remark}[theorem]{Remark}

\newcommand{\rr}{{\mathbb R}}
\DeclareMathOperator*{\argmin}{argmin} 
\newcommand*{\pdot}{\mathbin{\scalerel*{\boldsymbol\odot}{\circ}}}

\definecolor{tableblue}{RGB}{31, 78, 121}
\definecolor{headerblue}{RGB}{46, 116, 181}
\definecolor{lightgrey}{RGB}{242, 244, 247}

\title{Neural Compression of Long ADMM Trajectory for Multiparametric Quadratic Program}
\hypersetup{
  pdftitle={Neural Compression of Long ADMM Trajectory for Multiparametric Quadratic Program},
  pdfauthor={Liang Wu, Honghui Zheng, Bo Yang, Xu Yang, Yilin Mo, J\'an Drgo\v na},
}

\author{%
  Liang Wu$^{1,}$\thanks{Equal Contributions. $^{\dagger}$Corresponding Author.}\quad Honghui Zheng$^{1,*,\dagger}$\quad Bo Yang$^{2,*}$\quad  Xu Yang$^{2,*}$\quad Yilin Mo$^{2}$ \quad  J\'an Drgo\v na$^{1}$\\
  $^{1}$Johns Hopkins University\quad $^{2}$Tsinghua University\\
  \texttt{\{wliang14, hzheng39, jdrgona1\}@jh.edu}\\
  \texttt{\{yang-b21, yangx21\}@mails.tsinghua.edu.cn}, \texttt{ylmo@tsinghua.edu.cn}
}

\begin{document}

\maketitle

\begin{abstract}
Solving large-scale multiparametric quadratic programs (mpQPs) in real time often requires thousands of iterative optimization updates, creating a major computational bottleneck in model predictive control and learning-enabled systems. This paper proposes \emph{TQPNet}, a self-supervised neural compression framework that learns to compress long ADMM optimization trajectories into a compact neural warm start followed by a few refinement iterations. We derive a low-dimensional ADMM scheme operating on the primal--dual variables $(x,\beta)$ and show that its iteration admits an equivalent ReLU-layer representation, termed \emph{ADMM(ReLU)}. Building on this insight, TQPNet combines a multilayer perceptron with ReLU activations (\emph{MLP(ReLU)}) and a small number of unrolled \emph{ADMM(ReLU)} refinement layers. The network is trained offline using an optimality-condition-informed residual loss, eliminating the need for solver-generated labels. The proposed residual-based framework enables both self-supervised training and online solution certification through adaptive \emph{ADMM(ReLU)} refinement. Experiments on large-scale mpQPs and real-time MPC applications demonstrate that \emph{TQPNet} achieves high accuracy, strong generalization, and substantially faster inference than conventional optimization-based methods.
\end{abstract}

\vspace{-10pt}
\section{Introduction and Problem Statement}
\vspace{-10pt}
Quadratic programming (QP) is a fundamental optimization paradigm widely used across control, operations research, machine learning, finance, and related fields \cite{nocedal2006numerical,borrelli2017predictive}. Specifically, the widely used optimization-based control approach, model predictive control (MPC), can often be recast as a QP with time-varying parameters, referred to as a multiparametric QP (mpQP) \cite{bemporad2002explicit}. 

\begin{wrapfigure}[4]{r}{0.53\textwidth}  
\vspace{-10pt}
\rule{\linewidth}{0.6pt}
\begin{equation}\label{eqn_QP}
        \min_{x\in\rr^n}~ \frac{1}{2}x^\top Q x + q(p)^\top x,~\text{s.t.}~ Ax \geq c(p)
\end{equation}
\rule{\linewidth}{0.6pt}
\vspace{-5pt}
\end{wrapfigure}
This paper focuses on a convex mpQP problem as shown in \eqref{eqn_QP}, where $p\in\mathcal{P}\subset\rr^{n_p}$ denotes the parametric vector. The problem data, including $Q=Q^\top\in\rr^{n\times n}\succeq0$,  $A\in\rr^{m\times n}$, $q(p)\in\rr^{n}$ and $c(p)\in\rr^{m}$ dependent on $p$, and $x\in\rr^{n}$ denotes the decision vector. We assume that QP \eqref{eqn_QP} has an optimal solution set $\mathcal{X}^*$, namely the set $\{x\in\rr^n: Ax\geq c(p)\}$ is not empty. Note that this assumption is not restrictive, as infeasible instances can be equivalently handled via a slack-variable reformulation. Specifically, introducing slack variables $\xi \in \mathbb{R}^m$,  we consider the penalized problem $\min_{x,\xi}~\frac{1}{2}x^\top Q x+x^\top q(p)+\frac{\alpha}{2}\xi^\top \xi,\text{s.t. }Ax+\xi\geq c(p),~\xi\geq0$ where $\alpha>0$ is a penalization scalar, which can be written as the standard mpQP formulation \eqref{eqn_QP} by recasting $x\leftarrow\mathrm{col}(x,\xi)$, $Q\leftarrow\left[\begin{array}{cc}
    Q & 0_{n\times m} \\
    0_{m\times n} & \alpha I_m
\end{array} \right]$, $q(p)\leftarrow\left[\begin{array}{c}
     q(p)  \\
     0_m 
\end{array}\right]$, $A\leftarrow\left[\begin{array}{cc}
    A &  I_m\\
    0_{m\times n}& I_m
\end{array}\right]$, and $c(p)\leftarrow\left[\begin{array}{c}
     c(p)  \\
     0_m 
\end{array}\right]$. 

The mpQP \eqref{eqn_QP} arises from MPC formulations for both linear and nonlinear systems (e.g., via Koopman operator \cite{korda2018linear}). Despite the maturity of QP theory and algorithms \cite{wu2026EIQP,wu2026n3}, solving large-scale mpQPs in real time often requires thousands of iterative optimization updates, which remain a major computational bottleneck in modern control, robotics, and learning-enabled decision-making systems.

This challenge becomes even more pronounced in scenarios requiring massively parallel optimization, such as reinforcement-learning-MPC pipelines and differentiable control architectures, where a large number of mpQPs must be solved simultaneously and repeatedly. In such settings, there is a growing need for optimization methods that simultaneously achieve:
\textit{(i)} fast inference,
\textit{(ii)} scalability in both offline and online phases,
\textit{(iii)} batch-parallel execution,
\textit{(iv)} differentiability for end-to-end learning integration,
and \textit{(v)} trustworthy solution refinement with optimality certification.

\textbf{Related work.} Existing approaches for solving mpQPs can be broadly categorized into several classes. \emph{Online optimization-based solvers}, such as using interior-point \cite{nesterov1994interior,vanderbei1999loqo,wu2026EIQP,wu2025time,wachter2006implementation}, active-set \cite{goldfarb1983numerically,ferreau2014qpoases}, and first-order methods \cite{stellato2018osqp,lu2026practical,wu2026pi,patrinos2013accelerated}, offer convergence guarantees and optimality certificates, but do not exploit offline computation to accelerate online inference. \emph{Offline explicit region-based methods} \cite{bemporad2002explicit,tondel2003algorithm,herceg2013multi} rely on offline computation to precalculate the piecewise-affine solution mapping, enabling extremely fast online evaluation, but their offline construction suffers from the well-known combinatorial growth of critical regions, which quickly becomes prohibitive for high-dimensional problems. This limits their applicability to large-scale MPC. \emph{Supervised neural approaches} ~\cite{karg2020efficient,karg2018deep,lucia2020deep} learn mappings from problem parameters to optimal decisions using solver-generated labeled data. While enabling fast inference, they require expensive offline supervision and generally lack mechanisms for adaptive refinement or solution certification. More recently, \emph{learning-to-optimize and self-supervised approaches} \cite{chen2022learning,ichnowski2021accelerating,sambharya2024learning,donti2021dc,luken2024self,king2024metric,nguyen2025fsnet,tran2026huanet,wang2026parametric} have explored learning iterative optimization structures, warm-start strategies, or optimization dynamics without requiring labeled optimal solutions. Finally, \emph{differentiable optimization approaches} \cite{amos2017optnet,agrawal2019differentiable,butler2021efficient,besancon2023diffopt,pmlr-v139-paulus21a}, enable end-to-end gradient propagation through optimization layers, but often rely on implicit differentiation involving large Hessian-related linear systems, which can incur substantial computational and memory overhead.

Motivated by these challenges, this paper proposes \emph{TQPNet}, a self-supervised neural compression framework for mpQPs, illustrated in Fig.~\ref{fig_TQPNet}. The central idea is to learn to compress long ADMM algorithm trajectories into a compact neural warm start followed by only a few optimization refinement iterations.

\begin{figure}[!htbp]
    \centering
    \includegraphics[width=1\linewidth]{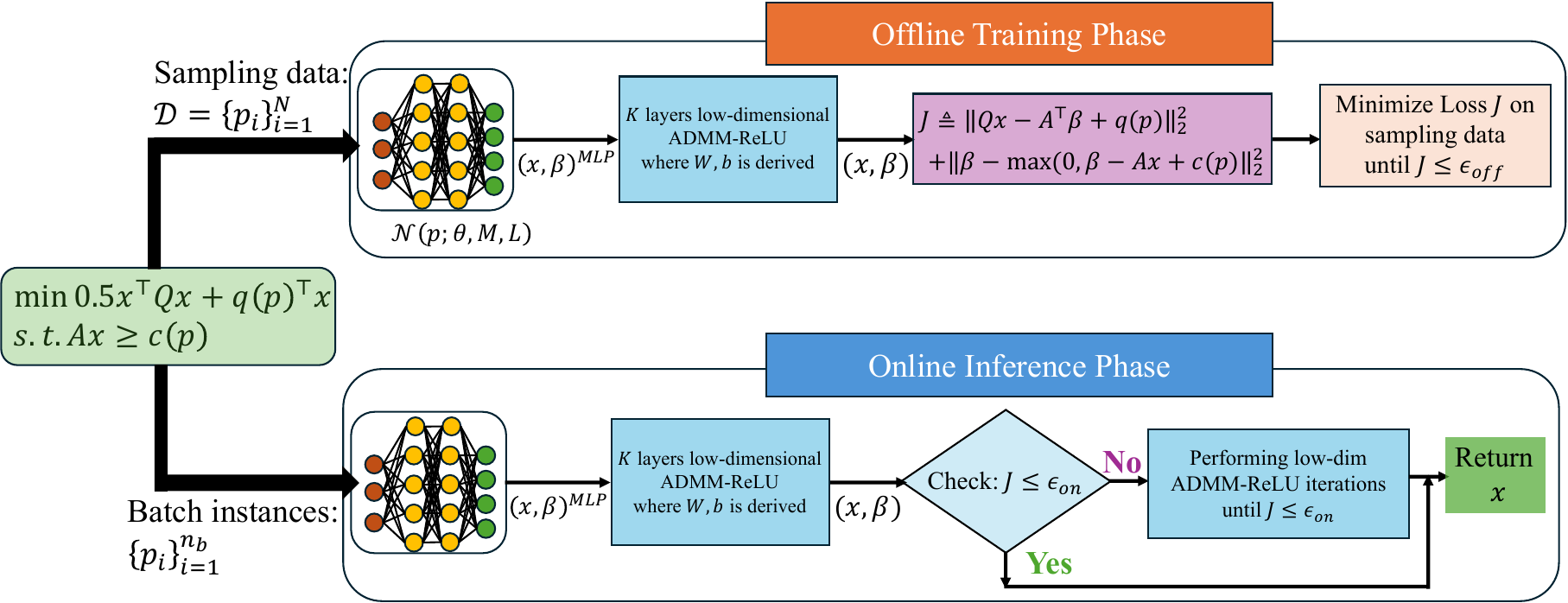}
    \caption{Architecture of \emph{TQPNet} based on the hybrid \emph{MLP(ReLU)}-\emph{ADMM(ReLU)} neural network.}
    \label{fig_TQPNet}
\end{figure}

To enable this formulation, we first derive a reduced-state ADMM iteration scheme operating directly on the primal--dual variables $(x,\beta)$. We further show that the resulting iteration admits an equivalent ReLU-layer representation, termed \emph{ADMM(ReLU)}, enabling a neural-friendly and naturally batch-parallel implementation. Building on this insight and motivated by the strong expressive power of deep multilayer perceptrons with ReLU activations (\emph{MLP(ReLU)}) for representing piecewise-affine structures, which naturally arise in mpQP solution mappings, TQPNet combines a multilayer perceptron with ReLU activations (MLP(ReLU)) and a small number of unrolled \emph{ADMM(ReLU)} refinement layers. The \emph{MLP(ReLU)} component learns a compact approximation of long optimization trajectories, while the \emph{ADMM(ReLU)} layers provide optimization-based refinement and adaptive certification. Rather than relying on solver-generated labels, the \emph{MLP(ReLU)} parameters are trained offline in a self-supervised manner using an optimality-condition-informed residual loss over sampled parameter instances. The proposed residual-based formulation provides a unified mechanism for both offline training supervision and online solution verification. When the predicted solution does not satisfy a prescribed optimality tolerance, additional \emph{ADMM(ReLU)} refinement iterations are triggered automatically to improve solution accuracy.

As a result, \emph{TQPNet} combines the efficiency of amortized optimization with the reliability of iterative optimization refinement. Extensive experiments demonstrate that the proposed framework can compress thousands of ADMM iterations into a compact \emph{MLP(ReLU)} followed by only a few \emph{ADMM(ReLU)} refinement steps, leading to substantial acceleration on large-scale mpQPs and real-time MPC applications while preserving strong solution accuracy and generalization performance.

\textbf{Contributions.} Our main contributions can be summarized as follows,
\begin{itemize}
    \item To the best of the authors’ knowledge, this is among the first works to propose a low-dimensional ADMM with lower per-iteration cost and empirically faster convergence than standard ADMM and show its equivalence to a ReLU layer, enabling a differentiable and batch-parallel \emph{ADMM(ReLU)} implementation.
    \item We propose \emph{TQPNet}, a hybrid \emph{MLP(ReLU)}–\emph{ADMM(ReLU)} architecture that learns to compress long optimization trajectories into compact neural warm starts followed by a few optimization refinement iterations. \emph{TQPNet} is an optimality-condition-informed self-supervised learning framework that eliminates the need for solver-generated labels while enabling adaptive online refinement and solution verification.
    \item More broadly, this work suggests a new paradigm for designing neural compressions of deep iterative algorithms and integrating them with the underlying optimization iterations to achieve faster online computation.
\end{itemize}

\section{Method}

\subsection{Preliminary: full-dimensional ADMM algorithm}
Different variable-splitting schemes lead to different ADMM algorithms. One standard choice is introducing an auxiliary variable $z = Ax\in\rr^m$
and QP \eqref{eqn_QP} can be reformulated as follows,
\begin{equation}
\min_{x,z}~\frac{1}{2} x^\top Q x + q(p)^\top x + \mathcal{I}_{\mathcal{Z}}(z), \text{ s.t. }  z-Ax=0
\end{equation}
\vspace{-6pt}
\begin{wrapfigure}[8]{r}{0.62\textwidth}  
\vspace{-10pt}
\rule{\linewidth}{0.6pt}
\vspace{-10pt}
\begin{subequations}\label{eqn_ADMM_bad}
    \begin{align}
        x^{k+1} &=\Big(Q+\rho A^\top A \Big)^{-1}\Big(-q(p)+\rho A^\top(z^k+u^k)\Big)\label{eqn_ADMM_bad_a}\\
        z^{k+1}&=\max\Big(c(p), Ax^{k+1}-u^k\Big)\\
        u^{k+1} &= u^k + z^{k+1} - Ax^{k+1}
    \end{align}
\end{subequations}
\rule{\linewidth}{0.6pt}
\vspace{-5pt}
\end{wrapfigure}
where $\mathcal{Z}=\{z\in\rr^m: z\geq c(p) \}$ and $\mathcal{I}_{\mathcal{Z}}(\cdot)$ is the indicator function for the set $\mathcal{Z}$. Then, its augmented Lagrangian (with the scaled-form Lagrangian variable) is given by $\mathcal{L}_{\rho}(x,z,u) = \frac{1}{2}x^\top Q x + q(p)^\top x +\mathcal{I}_{\mathcal{Z}}(z)+ \frac{\rho}{2}\|z-Ax+u\|_2^2$, where $u\in\rr^m$ denotes the scaled dual Lagrangian variable for the equality constraint. Based on Ref. \cite{boyd2011distributed}, the corresponding $(k+1)$-th ADMM update is shown in Eqn. \eqref{eqn_ADMM_bad}.
\begin{remark}\label{remark_data_requirement}
\textbf{(Requirements of Problem Data for Well-posedness of ADMM).} The $x$-update \eqref{eqn_ADMM_bad_a} is well-defined if $Q+\rho A^\top A\succ0$, or equivalently, $\operatorname{null}(Q)\cap\operatorname{null}(A)=\{0\}$,
which ensures the uniqueness of the $x$-update \eqref{eqn_ADMM_bad_a}.
\end{remark}
\begin{lemma}\label{lemma_ADMM_ReLU_3}
In the ADMM iteration \eqref{eqn_ADMM_bad}, the $x$-, $z$-, and $u$-updates are \textbf{performed sequentially}. To improve parallelizability and get faster computation speed, we can rewrite \eqref{eqn_ADMM_bad} as the following \textbf{ReLU-like} layer,
\begin{equation}\label{eqn_ADMM_ReLU_like}
    \begin{bmatrix}
        x^{k+1}\\
        z^{k+1}\\
        u^{k+1}
    \end{bmatrix}
    =
    \mathrm{ReLU}_{z\geq c(p)}\!\left(
    W_3
    \begin{bmatrix}
        x^k\\
        z^k\\
        u^k
    \end{bmatrix}
    +b_3\right),   
\end{equation}
where 

$W_3\triangleq
\begin{bmatrix}
    -\rho H(A^\top A) & 2\rho H A^\top & \rho H A^\top\\[1mm]
    A-\rho A H(A^\top A)& 2\rho A H A^\top-I & \rho A H A^\top-I\\[1mm]
    -A & I & I
    \end{bmatrix},$ $b_3\triangleq\begin{bmatrix}
    -Hq(p)\\[1mm]
    -AHq(p)\\[1mm]
    0
\end{bmatrix}$, and $\mathrm{ReLU}_{z\geq c(p)}
\left(
\begin{bmatrix}
    x\\[1mm]
    z\\[1mm]
    u
\end{bmatrix}
\right)\triangleq
\begin{bmatrix}
    x\\[1mm]
    \max(c(p),z)\\[1mm]
    u
\end{bmatrix},$ with $H\triangleq\big(Q+\rho A^\top A\big)^{-1} $.
\end{lemma}
\begin{proof}
    See Appendix \ref{appendix_sec_lemma_ADMM_ReLU_3}.
\end{proof}

\begin{remark}\label{remark_cost_per_iteration}
The ADMM-ReLU formulation \eqref{eqn_ADMM_ReLU_like} is \textbf{full-dimensional} over $(x,z,u)\in\mathbb{R}^{n+2m}$, with computational complexity $O((n+2m)^2)$, as demonstrated in \cite{bishop2024relu}.
\end{remark}

To simultaneously eliminate the data requirements noted in Remark~\ref{remark_data_requirement} and reduce the per-iteration computational cost of ADMM (i.e., the ADMM-ReLU layer) discussed in Remark~\ref{remark_cost_per_iteration}, we propose a computationally efficient \textbf{low-dimensional} ADMM-ReLU scheme.

\subsection{Low-dimensional ADMM-ReLU layer}
Before presenting the proposed low-dimensional ADMM-ReLU, we introduce a decomposition of $A$, termed the \textit{Canonical Outer-Product Expansion} (COPE), as an analytical tool. Notably, the resulting low-dimensional formulation does not rely on the COPE decomposition and does not perform the decomposition in practice.
\begin{definition}
    \textbf{(COPE)}. Let $A\in\mathbb{R}^{m\times n}$ be a matrix with $\mathrm{nnz}(A)=o$ nonzero entries. Let $\{(i_k,j_k,A_{i_kj_k})\}_{k=1}^{o}$ be any enumeration of the nonzero entries of $A$, where $A_{i_kj_k}\neq 0$, $i_k\in\{1,\dots,m\}$, and $j_k\in\{1,\dots,n\}$. Then $A$ admits the \emph{canonical outer-product expansion}
    \begin{equation}\label{eq:canonical_rank_one_expansion}
    A \;=\; \sum_{k=1}^{o}A_{i_kj_k}\, e_{i_k}^m (e_{j_k}^n)^\top.
    \end{equation}
    where $\{e_i^m\}$ and $\{e_j^n\}$ are the canonical basis vectors in $\mathbb{R}^{m}$ and $\mathbb{R}^{n}$, respectively.
    Denote 
    \begin{equation}\label{eqn_cope}
        U \triangleq \sum_{k=1}^o A_{i_k,j_k} e_{i_k}^m (e_{k}^o)^\top\in\mathbb{R}^{m\times o},\quad
        V \triangleq \sum_{k=1}^o e_{j_k}^n (e_{k}^o )^\top\in\mathbb{R}^{n\times o}
    \end{equation}
    where $\{e_k^o\}$ are the canonical basis vectors in $\mathbb{R}^{o}$, then
    \[
    UV^\top=A_{i_1,j_1}e_{i_1}^m (e_1^o)^\top e_1^o (e_{j_1}^n)^\top + \cdots + A_{i_o,j_o}e_{i_o}^m (e_o^o)^\top e_o^o (e_{j_o}^n)^\top= A,    
    \]
    which is the matrix form of COPE.
\end{definition}
\begin{lemma}\label{lemma_D_U_D_V}
    Let $A\in\rr^{m\times n}$ be a (sparse) matrix each of whose rows and columns contains at least one nonzero entry, and let $(U,V)$ defined in \eqref{eqn_cope} be its COPE representation. Then, the matrices $D_U\triangleq UU^\top\in\mathbb{R}^{m\times m}$ and $D_V\triangleq VV^\top\in\mathbb{R}^{n\times n}$ are positive diagonal matrices, with
    \begin{equation}
        (D_U)_{ij} =
        \begin{cases}
            A_{i,:}A_{i,:}^\top, & i=j,\\
            0,   & i\neq j,
        \end{cases}
        ~ i,j=1,\dots,m,\quad
        (D_V)_{ij} =
        \begin{cases}
            \mathrm{nnz}(A_{:,i}), & i=j,\\
            0,   & i\neq j,
        \end{cases}
        ~ i,j=1,\dots,n,
    \end{equation}
    which indicates that the $i$-th diagonal value of $(D_U)$ is the sum of the squares of all non-zero elements in $A_{i,:}$ (the $i$-th row of $A$), and the $i$-th diagonal value of $(D_V)$ is the number of non-zero elements in $A_{:,i}$ (the $i$-th column of $A$).
\end{lemma}
\begin{proof}
    See Appendix \ref{appendix_sec_lemma_D_U_D_V}.
\end{proof}
Note that the subsequent analysis does not require a full COPE decomposition, but only the inexpensive and trivial computation of $D_V$ and $D_U$. Then, using the COPE decomposition $A=UV^\top$ and introducing the auxiliary variable $y\in\rr^o$: $y=V^\top x$, we can equivalently reformulate QP \eqref{eqn_QP} as the following optimization:
\begin{equation}
        \min_{x,y}~ \frac{1}{2}x^\top Q x + q(p)^\top x + \mathcal{I}_{\mathcal{Y}}(y),~\text{s.t.}~ y-V^\top x  =0
\end{equation}
where the set $\mathcal{Y}\triangleq\{y\in\rr^o:Uy\geq c(p)\}$ and $\mathcal{I}_{\mathcal{Y}}(\cdot)$ is the indicator function for the set $\mathcal{Y}$. 
\begin{wrapfigure}[8]{r}{0.64\textwidth}  
\rule{\linewidth}{0.6pt}
\vspace{-10pt}
\begin{subequations}\label{eqn_ADMM}
    \begin{align}
        y^{k+1}&=\argmin_{y} \frac{1}{2}\|y-(V^\top x^k-u^k)\|_2^2,~\text{s.t.}~Uy\geq c(p)\label{eqn_ADMM_y}\\
        x^{k+1} &=(Q+\rho D_V)^{-1}\left(\rho V(y^{k+1}+u^k)- q(p)\right)\\
        u^{k+1} &= u^k + y^{k+1}-V^\top x^{k+1} 
    \end{align}
\end{subequations}
\rule{\linewidth}{0.6pt}
\vspace{-5pt}
\end{wrapfigure}
Then, its augmented Lagrangian (with the scaled-form Lagrangian variable) is given by $\mathcal{L}_{\rho}(x,y,u)=\frac{1}{2}x^\top Q x + q(p)^\top x + \mathcal{I}_{\mathcal{Y}}(y) + \frac{\rho}{2}\| y-V^\top x  +u\|_2^2$, where $u\in\rr^o$ denotes the scaled dual Lagrangian variable for the equality constraint. Based on Ref. \cite{boyd2011distributed}, the corresponding $(k+1)$-th ADMM update is shown in \eqref{eqn_ADMM}.

\begin{lemma}\label{lemma_closed_form}
    The $y$-update \eqref{eqn_ADMM_y} admits the following closed-form solution
    \begin{subequations}\label{eqn_y_update_closed_form}
        \begin{align}
            \beta^{k}&=\max\left(0,D_{U}^{-1}(-Ax^k+Uu^k+c(p))\right)\\
                y^{k+1}&=V^\top x^k-u^k+U^\top\beta^{k}
        \end{align}
    \end{subequations}
\end{lemma}
\begin{proof}
    See Appendix \ref{appendix_sec_lemma_closed_form}.
\end{proof}

\begin{theorem}\label{theorem_low_dim_ADMM_ReLU}
    \textbf{(Low-dimensional ADMM-ReLU Layer).} The COPE-based ADMM iteration \eqref{eqn_ADMM} can be equivalently rewritten as a low-dimensional ADMM-ReLU layer as follows,
    \begin{equation}\label{eqn_ADMM_ReLU_2}
    \left[\begin{array}{c}
        x^{k+1}  \\
         \beta^{k+1} 
    \end{array}\right]=\mathrm{ReLU}_{\beta\geq0}\left( W_2\left[\begin{array}{c}
        x^{k}  \\
         \beta^{k} 
    \end{array}\right]+b_2\right)
    \end{equation}
    where $H_V\triangleq\left(Q+\rho D_V\right)^{-1}$, $\mathrm{ReLU}_{\beta\geq0}\left(\left[\begin{smallmatrix} x \\ \beta \end{smallmatrix}\right]\right)\triangleq\left[\begin{smallmatrix} x \\ \max(0,\beta) \end{smallmatrix}\right]$, and
    \[
    \begin{aligned}
    W_2&\triangleq\left[\begin{array}{cc}
        \rho H_VD_V &  \rho H_VA^\top\\
        -D_U^{-1}A(2\rho H_VD_V-I)& I-2\rho D_U^{-1}AH_VA^\top
    \end{array}\right],\\
    b_2&\triangleq\left[\begin{array}{c}
         -H_Vq(p)  \\
        2 D_U^{-1}AH_Vq(p)  + D_U^{-1}c(p)
    \end{array}\right],
    \end{aligned}
    \]
    with computational complexity $O((n+m)^2)$ per layer or iteration.
\end{theorem}

\begin{proof}
    See Appendix \ref{appendix_sec_low_dim_ADMM_ReLU}. 
\end{proof}

In the proposed low-dimensional ADMM-ReLU \eqref{eqn_ADMM_ReLU_2}, $(x,\beta)$ can be represented as a single affine-ReLU operator evaluation, enabling a parallelizable and neural-friendly implementation that naturally supports batched computations. Furthermore, the unrolled form of \eqref{eqn_ADMM_ReLU_2} is inherently differentiable, allowing seamless end-to-end integration with learning-based modules. However, the ADMM iteration scheme may require many iterations when applied to ill-conditioned mpQPs. This motivates the development of a compact neural proxy to learn the solution–parameter mapping of mpQP in the offline phase.

\subsection{Self-supervised Learning of Trustworthy Neural Proxy}
A well-known result is that the solution of the mpQP \eqref{eqn_QP} partitions the parameter vector space into critical regions (CRs), where a CR is the set of all parameter vectors $p$ for which a certain combination of constraints is active at the optimizer of the multiparametric QP \eqref{eqn_QP}. For each CR, the optimizer $x^*$
 is a piecewise affine (PWA) function of $p$, see \cite{bemporad2002explicit}, i.e.,
\begin{equation}\label{eqn_PWA}
    \textit{Explicit PWA: } x^* = \begin{cases} 
        K_1p+ f_1, & \text{if } p \in \mathcal{R}_1\triangleq\{p\in\rr^{n_p}| G_1p\leq g_1 \} \\ 
        \quad\vdots,&\vdots \\
        K_{n_r}p + f_{n_r}, & \text{if } p \in \mathcal{R}_{n_r} \triangleq\{p\in\rr^{n_p}| G_{n_r}p\leq g_{n_r}\}
    \end{cases}
\end{equation}
One major drawback of the \textit{explicit PWA} representation \eqref{eqn_PWA} of the multiparametric QP \eqref{eqn_QP} is that the number of regions $n_r$ required for an exact representation can grow exponentially with the problem dimension, leading to large memory requirements and rendering \textit{explicit PWA} approaches \eqref{eqn_PWA} non-scalable.

\subsubsection{Hybrid \emph{MLP(ReLU)}-\emph{ADMM(ReLU)} neural network structure}
A scalable and memory-efficient alternative is a deep MLP with ReLU activation functions, motivated by the expressive power results in \cite{montufar2014number}. In this paper, we adopt the following \emph{MLP(ReLU)}-\emph{ADMM(ReLU)} parameterization as a proxy for the solution of multiparametric QP \eqref{eqn_QP}:
\begin{equation}\label{eqn_DNN_proxy}
    \left(\begin{array}{c}
         x^*\\
         \beta^*\\
    \end{array}\right)
    =\mathrm{ADMM\text{-}ReLU}_{K,\rho}\Big(\mathcal{N}(p;\theta,M,L)\Big)
\end{equation}
where, $\mathcal{N}(p;\theta,M,L)=F_{L+1}\circ G_L\circ F_L\circ \cdots \circ G_1\circ F_1(p):\mathbb{R}^{n_p}\rightarrow\mathbb{R}^{n}$ with $F_l(\xi_{l-1}) = W_l \xi_{l-1} + b_l$, and $G_l(\xi)=\max(0,\xi),$ for all $l=1,\dots,L$, (where $L\geq 2$); $\mathrm{ADMM\text{-}ReLU}_{K,\rho}(\cdot)$ denotes the proposed low-dimensional ADMM-ReLU formulation in \eqref{eqn_ADMM_ReLU_2}, with a fixed number of iterations (or layers) $K$ and a penalty parameter $\rho$ (can be learnable). The learnable parameter is given by $\theta=\{W_1,b_1,\dots,W_{L+1},b_{L+1}\}$, where $M$ denotes the width of each hidden layer and $L$ denotes the number of hidden layers.
\begin{lemma}[{see \cite[Thm. 4, Cor. 5]{montufar2014number}}]\label{lemma_DNN}
Every neural network $\mathcal{N}(p; \theta, M, L)$ with ReLUs as activation functions and $M\geq n_p$ represents a PWA function. In addition, a lower bound on the maximal number of affine regions that the neural network represents is given by the following expression:
\[
    \left(\prod_{l=1}^{L-1}\left\lfloor \frac{M}{n_{p}}\right\rfloor^{n_{p}}\right) \sum_{j=0}^{n_{p}}\left(\begin{array}{c}
        L \\
        j
    \end{array}\right)\sim\Omega\left(\left(\frac{M}{n_p}\right)^{(L-1)n_p}M^{n_p} \right).
\]
\end{lemma}
Lemma \ref{lemma_DNN} shows that the number of linear regions of a MLP grows exponentially with $L$ and polynomially with $M$. Furthermore, the bound in Lemma \ref{lemma_DNN} can be slightly improved when $M\geq 3n_p$; see \cite{serra2018bounding}.

\subsubsection{Self-supervised learning loss design}
Rather than using solver-generated data to build a supervised loss, this paper introduces an \textbf{optimality-condition-informed neural network} that constructs the loss within a self-supervised learning framework. Define
\begin{equation}\label{eqn_r_definition}
    r(x,\beta;p) \triangleq \begin{bmatrix}
       Qx- A^\top \beta + q(p)\\[1mm]
        \beta-\max\left(0,\beta-Ax+c(p) \right)
    \end{bmatrix}
\end{equation}
where it is immediate that $\|r(x,\beta;p)\|_2^2$ is piecewise-quadratic and convex on the solution manifold.
\begin{theorem}\label{theorem_main}
    For any $p\in\mathcal{P}$, the multiparametric QP \eqref{eqn_QP} admits a nonempty solution set $\mathcal{S}(p)$. Then, there exists a constant $\kappa(p)>0$ such that $\mathrm{dist}\big((x,\beta),\mathcal{S}(p)\big)\leq\kappa(p)\left\|r(x,\beta;p)\right\|_2$ for all $(x,\beta)$ in a neighborhood of $\mathcal{S}(p)$.
    Consequently, minimizing the \textbf{optimality-condition-informed} residual norm $\|r(x,\beta;p)\|_2$ drives the $(x,\beta)$
    toward the solution set of the multiparametric QP \eqref{eqn_QP}.
\end{theorem}
\begin{proof}
    See Appendix \ref{appendix_sec_theorem_main}.
\end{proof}

By Theorem~\ref{theorem_main}, substituting the \emph{MLP(ReLU)}-\emph{ADMM(ReLU)}-parameterized proxy \eqref{eqn_DNN_proxy} into the residual $\|r(x,\beta;p)\|_2^2$ naturally induces a self-supervised learning framework for training the neural proxy. Because the $\ell_2$-norm of the optimality condition residual may have several orders of magnitude across different samples, in practice, we design the following training loss,
\begin{equation}
    \mathcal{L}(\theta)
    =\mathbb{E}_{p\sim\mathcal{D}}
    \left[
    \log\left(\left\|
    r\Big(\mathrm{ADMM\text{-}ReLU}_{K,\rho}\big(\mathcal{N}(p;\theta,M,L);p\big);p\Big)
    \right\|_2^2 + 1\right)
    \right].
\end{equation}
Minimizing $\mathcal{L}(\theta)$ is theoretically consistent with driving the prediction of the \emph{MLP(ReLU)}-\emph{ADMM(ReLU)}-parameterized proxy \eqref{eqn_DNN_proxy} toward the optimal primal-dual solution. The dataset $\mathcal{D}$ can be constructed by sampling the points $\{p_i\}_{i=1}^N$ uniformly over the domain $\mathcal{P}$. However, uniform sampling often provides insufficient coverage near the boundary of $\mathcal{P}$, even though these regions are particularly important for characterizing the behavior of the loss function and improving the accuracy of the surrogate model close to the boundary. To obtain better boundary coverage, this paper adopts a projected sampling strategy as shown in Algorithm \ref{alg_projected_sampling}. The unified algorithm for offline training and online inference with certification is presented in Algorithm~\ref{alg_offline_online}.

\section{Numerical Experiments}\label{sec_numerical_experiments}
\vspace{-10pt}
We evaluate \emph{TQPNet} on four increasingly complex experiments on a single RTX 3080 (per-experiment unrolling depth $K_{\text{train}}$ in Appendix~\ref{app:ablation_impl}). Code: \url{https://github.com/SOLARIS-JHU/TQPNet}.

\vspace{-10pt}
\subsection{Advantages of \emph{MLP(ReLU)}-\emph{ADMM(ReLU)}}
\vspace{-10pt}

\textbf{Empirically faster convergence of low-dimensional ADMM-ReLU.}
We evaluate the convergence behavior of the proposed low-dimensional ADMM-ReLU \eqref{eqn_ADMM} against the standard full-dimensional ADMM \eqref{eqn_ADMM_bad} on $100$ independently generated QP instances with dimensions $n=500$ and $m=1000$. The constraint matrix $A\sim\mathrm{scale}\times\mathcal{U}([0,1]^{m\times n})$ is constructed as a sparse matrix with density $0.1$ and the right-hand side is sampled as $b \sim \mathrm{scale}\times\mathcal{U}([0,1]^m)$. Fig.~\ref{fig2} presents the convergence statistics on three different $\mathrm{scale}=\{2,4,6\}$ settings across problem instances, including the median convergence trajectory and the corresponding $5\%\text{--}95\%$ percentile band over the $100$ randomly generated QPs. The results indicate that the proposed low-dimensional ADMM-ReLU \eqref{eqn_ADMM} consistently achieves faster convergence behavior than the standard full-dimensional ADMM \eqref{eqn_ADMM_bad} across problem instances. We \textbf{do not} aim to demonstrate that the proposed low-dimensional ADMM is \textbf{universally faster} than the standard full-dimensional ADMM, since both approaches may still require a large number of iterations in certain cases. Rather, this limitation motivates us to develop an offline self-supervised MLP(ReLU) neural compression of long ADMM iterations to achieve faster online inference.
\begin{figure}[h!]
\vspace{-10pt}
\centering
\begin{subfigure}{0.32\linewidth}
    \centering
    \includegraphics[width=\linewidth]{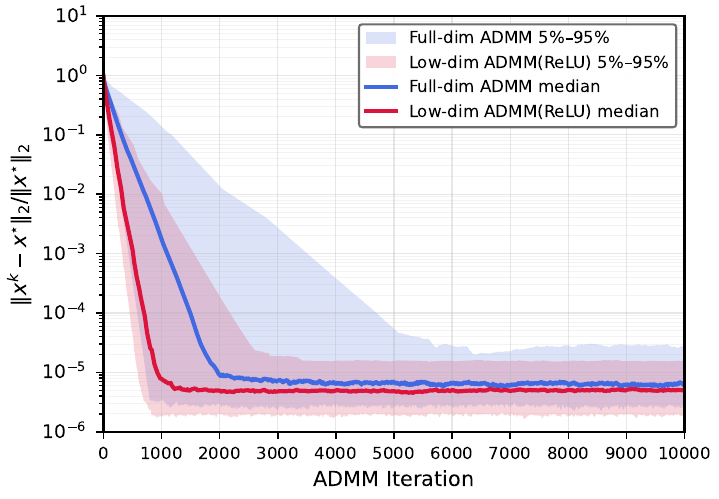}
    \caption{$\mathrm{scale}=2$}
\end{subfigure}
\hfill
\begin{subfigure}{0.32\linewidth}
    \centering
    \includegraphics[width=\linewidth]{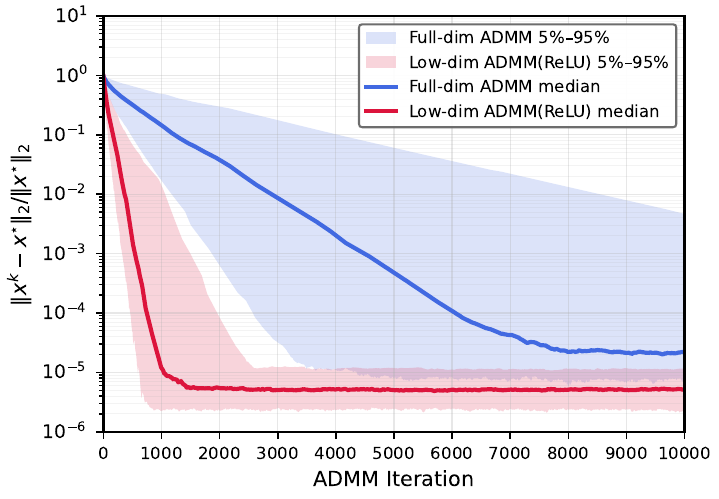}
    \caption{$\mathrm{scale}=4$}
\end{subfigure}
\hfill
\begin{subfigure}{0.32\linewidth}
    \centering
    \includegraphics[width=\linewidth]{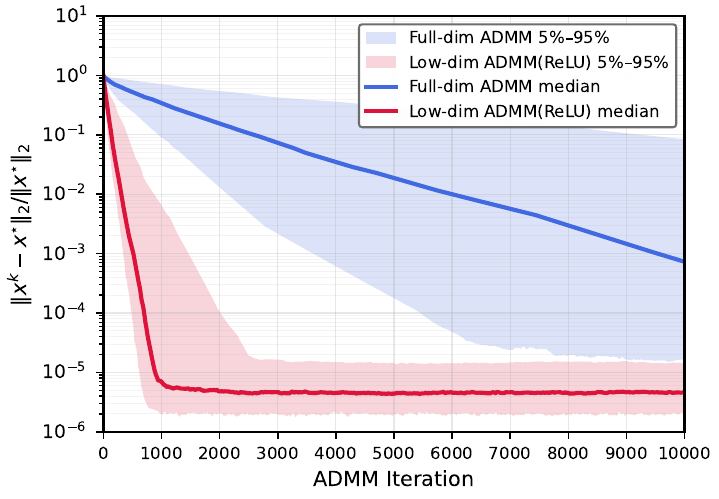}
    \caption{$\mathrm{scale}=6$}
\end{subfigure}
\vspace{-0.2cm}
\caption{Convergence statistics over $100$ random QPs for the low-dimensional ADMM-ReLU and the full-dimensional ADMM.}
\label{fig2}
\end{figure}

\vspace{-12pt}
\textbf{\emph{MLP(ReLU)} compresses long \emph{ADMM(ReLU)} iterations.}
Even after the low-dimensional reformulation, a typical mpQP at
$n{=}100$, $m{=}200$, $n_{p}{=}20$ still requires $4729$ cold-start
\emph{ADMM(ReLU)} iterations to reach $\varepsilon{=}10^{-7}$
(Fig.~\ref{fig:ablation}b).
We amortize this overhead into a learned warm start: an \emph{MLP(ReLU)} maps
$p \mapsto (x_{0},\beta_{0})$, trained end-to-end through
$K_{\text{train}}$ unrolled \emph{ADMM(ReLU)} layers under the KKT loss.
The depth $K_{\text{train}}$ governs whether back-propagation sees
warm-start error or residual ADMM-convergence noise: shallow unrolling
($K_{\text{train}}{\leq}10$) stays in the latter regime and fails to
train, whereas $K_{\text{train}}{=}25$ crosses into the former and
saturates the achievable accuracy
(Fig.~\ref{fig:ablation}a; the required $K_{\text{train}}$ grows with
$n$, Appendix~\ref{app:ablation}).
The trained warm start then cuts the inference budget by up to $1.8\times$
(Fig.~\ref{fig:ablation}b, Tab.~\ref{tab:ablation}).
We denote this \emph{MLP(ReLU)}-\emph{ADMM(ReLU)} hybrid~\eqref{eqn_DNN_proxy} as
\emph{TQPNet}; the supervised baseline \emph{LQPNet}
(Label-supervised QP Net) shares the same MLP backbone but is trained directly on OSQP labels via MSE.
\vspace{-10pt}
\begin{figure}[!htbp]
\centering
\begin{minipage}[t]{0.47\linewidth}
\centering
\includegraphics[width=\linewidth]{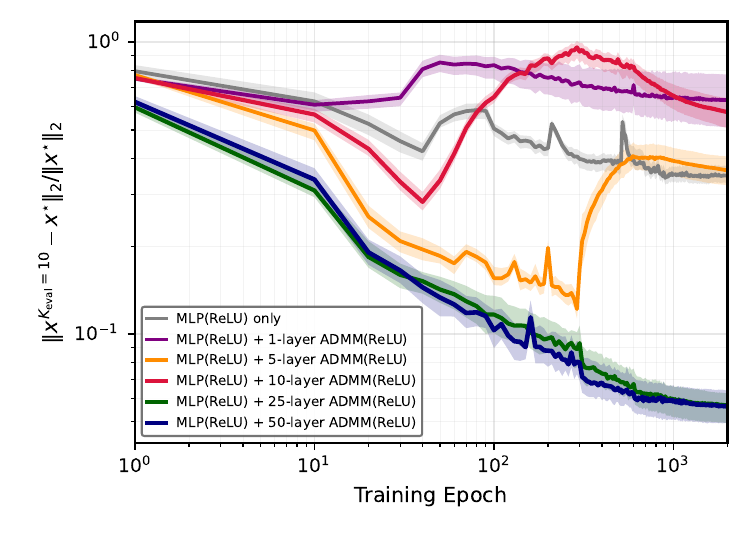}\\
{\small (a) Training.}
\end{minipage}
\hfill
\begin{minipage}[t]{0.47\linewidth}
\centering
\includegraphics[width=\linewidth]{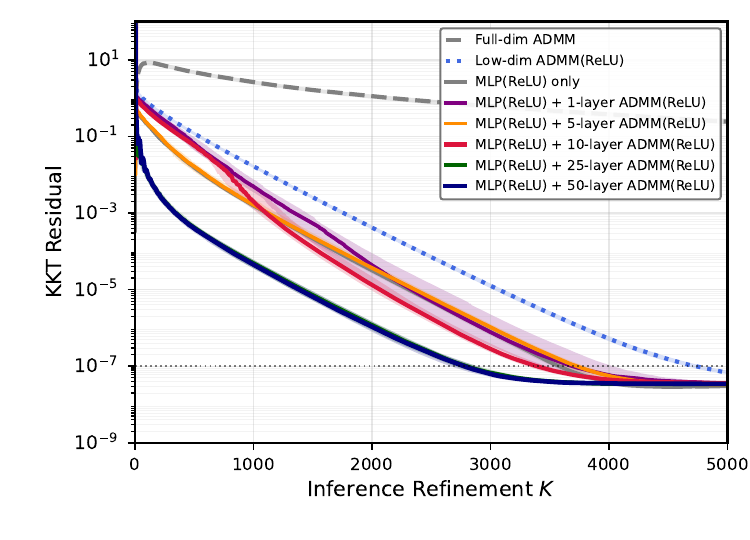}\\
{\small (b) Inference.}
\end{minipage}
\caption{Ablation of $K_{\text{train}}$.
(a) Test relative error vs.\ epoch at fixed $K_{\mathrm{eval}}{=}10$;
(b) post-training KKT residual vs.\ refinement steps $K$.
$K_{\text{train}}{\leq}10$ fails to compress; $K_{\text{train}}{\geq}25$ stabilizes training and accelerates inference $1.7$--$1.8\!\times$ over \emph{ADMM(ReLU)}.}
\label{fig:ablation}
\end{figure}

\subsection{Generalizability and Scalability}
\begin{wrapfigure}[13]{r}{0.5\textwidth}
    \vspace{-20pt}
    \centering
    \includegraphics[width=1\linewidth]{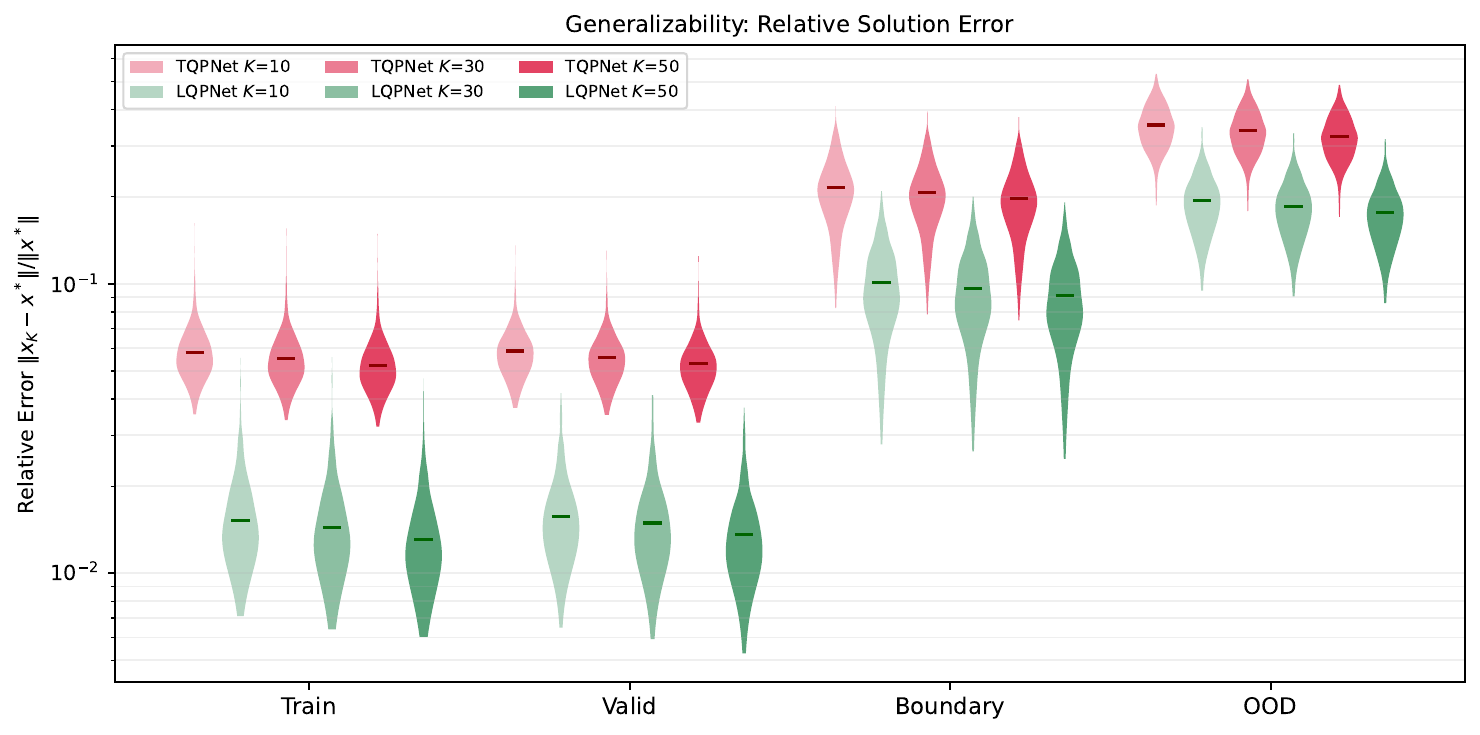}
    \vspace{-0.55cm}
    \caption{Relative solution error across parameter sets at $n{=}100$, $m{=}200$. \emph{LQPNet} is more accurate in distribution, but the TQP--LQP gap shrinks under distribution shift.}
    \label{fig:generalizability}
\end{wrapfigure}

\textbf{Out-of-distribution robustness.}
A key concern for learned solvers is their robustness under
distribution shift.
We train \emph{TQPNet} and \emph{LQPNet} on $n{=}100$, $m{=}200$, $n_p{=}20$
random mpQPs with parameters sampled from $p \sim
\mathcal{U}[-1,1]^{20}$, and evaluate on progressively shifted
ranges (Train/Valid: $\mathcal{U}[-1,1]$, Boundary:
$\mathcal{U}[-2,2]$, OOD: $\mathcal{U}[-3,3]$; $500$ instances
each).
In distribution, \emph{LQPNet} is more accurate (\emph{TQPNet} $5.8\!\times\!10^{-2}$ vs.\ \emph{LQPNet}
$1.5\!\times\!10^{-2}$ at $K{=}10$; Table~\ref{tab:generalizability}),
but the TQP--LQP gap narrows monotonically under shift, from $3.9\!\times$ in-distribution to $1.8\!\times$ on OOD
(Fig.~\ref{fig:generalizability}): \emph{TQPNet}'s KKT-residual training learns features more invariant to parameter shift than direct regression on $x^{\star}$.
\emph{TQPNet} additionally outputs the dual variable $\beta$ as a per-instance KKT certificate and trains without OSQP labels.

\begin{wrapfigure}[24]{r}{0.4\textwidth} 
    \vspace{-5pt}
    \begin{minipage}[b]{0.4\textwidth}
        \centering
        \includegraphics[width=\linewidth]{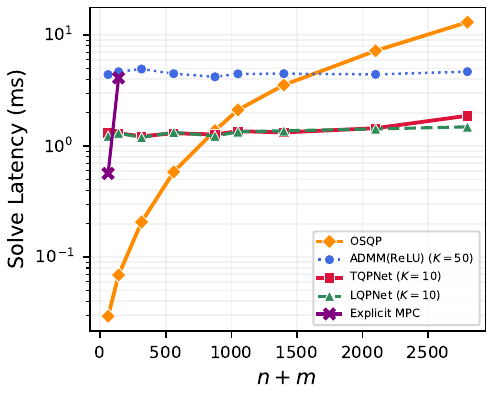}
        \vspace{-0.55cm}
    \end{minipage}
    \hfill
    \begin{minipage}[b]{0.4\textwidth}
        \centering
        \includegraphics[width=\linewidth]{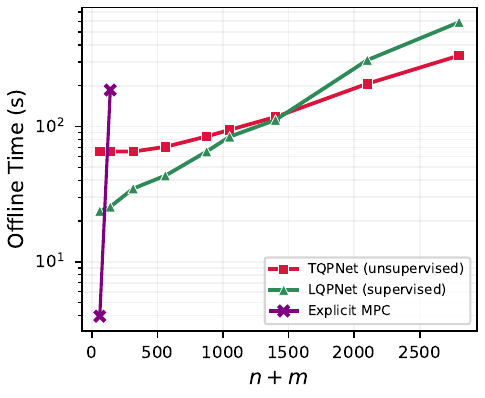}
        \vspace{-0.55cm}
    \end{minipage}
    \vspace{-15pt}
    \caption{(a) Solve latency. (b) Offline time.}
    \label{fig:scalability}
\end{wrapfigure}
\textbf{MPC scalability.} We benchmark MPC-structured QPs generated following Ref.~\cite{wang2009fast} across nine scales (S0--S8, $n{+}m = 60$--$2800$) using closed-loop simulations over $T{=}200$ steps. Fig.~\ref{fig:scalability}(a) shows that \emph{TQPNet} ($K{=}10$) maintains near-constant latency ($\sim$1.3\,ms) across all problem sizes, whereas OSQP exhibits significant growth from $0.03$\,ms to $13$\,ms as the scale increases. The crossover occurs at S4 ($n{+}m{=}875$), beyond which \emph{TQPNet} becomes faster, achieving a $7\times$ speedup at S8 ($n{+}m{=}2800$).

Explicit MPC (using MPT3 \cite{herceg2013multi}) already requires 3123 regions and incurs $4.1$\,ms latency at S1 ($n{+}m{=}140$), making it slower than \emph{TQPNet} at small scales and infeasible beyond S2 due to exponential region growth. Table~\ref{tab:scalability_mpc} further shows that both \emph{TQPNet} and \emph{LQPNet} match OSQP in tracking RMSE and closed-loop cost across all scales, while \emph{ADMM(ReLU)} with cold initialization ($K{=}50$) exhibits a noticeable performance gap, highlighting the importance of learned warm starts.

Fig.~\ref{fig:scalability}(b) compares offline preparation costs. \emph{TQPNet}’s unsupervised training pipeline avoids expensive OSQP label generation and becomes more efficient than \emph{LQPNet}’s supervised approach for $n{+}m \geq 2100$, demonstrating superior scalability in both online inference and offline training.

\subsection{Application to Quadrupedal Locomotion}\label{ssec:quadruped}
\vspace{-5pt}
Complementing the single-UAV experiment in Appendix~\ref{app:appendix_uav}, we now consider quadrupedal locomotion, where many QPs must be solved per control cycle.
Quadrupedal locomotion control is formulated via convex MPC \cite{di2018dynamic} running at $50\,\mathrm{Hz}$ over a prediction horizon of 16 steps with sampling period $\Delta t=0.05\,\mathrm{s}$. At each control cycle, the MPC problem is parameterized by $n_p = 18$ parameters, including the body roll--pitch--yaw angles, center-of-mass height, desired linear and angular velocities, and gait-related parameters (e.g., contact schedule and swing-phase indicators), and reduces to a QP with $n = 192$ decision variables and $m = 384$ inequality constraints. \emph{TQPNet} is trained offline on QP instances sampled from the locomotion parameter distribution over $\mathcal{P} \subset \mathbb{R}^{18}$, without any environment interaction or solver-generated labels, with training completing in under $5$ minutes on a single RTX 3080. The trained proxy is subsequently deployed in massively parallel simulation using Isaac Gym \cite{makoviychuk2021isaac}, in which each simulated robot corresponds to an independent QP instance that is solved concurrently within a single control cycle.

\begin{wrapfigure}[22]{r}{0.35\textwidth} 
    \centering
    \vspace{-10pt}
    \begin{minipage}[b]{0.35\textwidth}
        \centering
        \includegraphics[width=\linewidth]{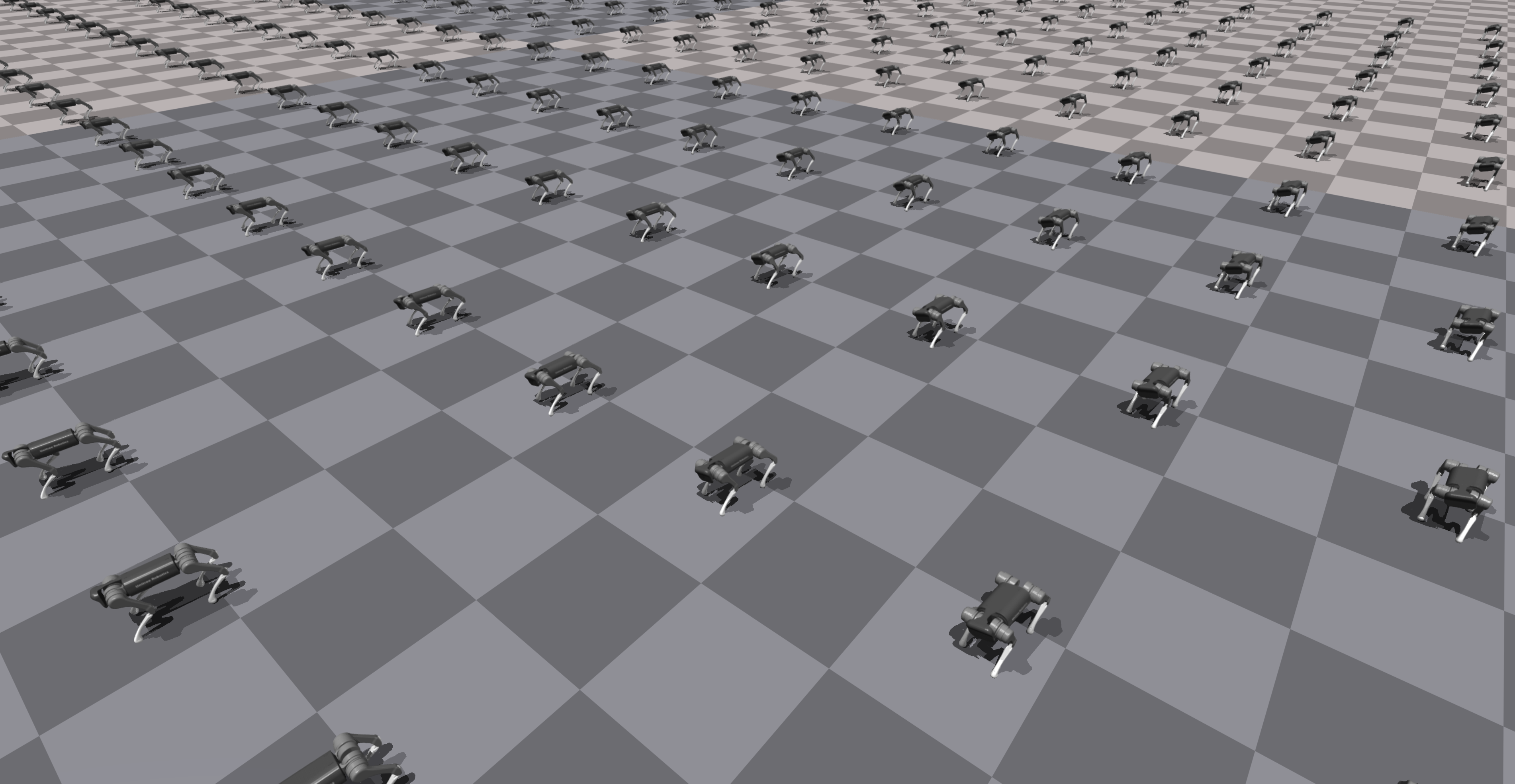}
        \vspace{-0.55cm}
    \end{minipage}
    \hfill
    \begin{minipage}[b]{0.35\textwidth}
        \centering
        \includegraphics[width=\linewidth]{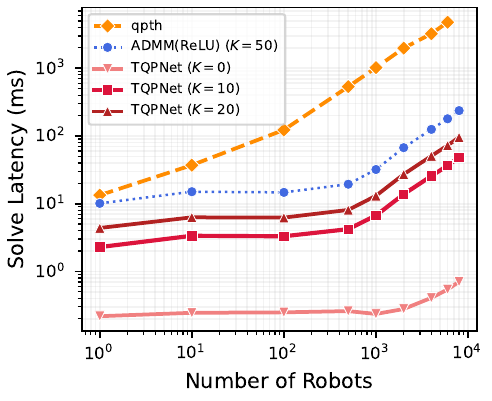}
        \vspace{-0.55cm}
    \end{minipage}
    \vspace{-15pt}
    \caption{(a) 4096 robots concurrently controlled by \emph{TQPNet}-based MPC in Isaac Gym. (b) Solve latency vs.\ number of parallel environments.}
    \label{fig:quadruped_batch}
\end{wrapfigure}
Fig.~\ref{fig:quadruped_batch} reports the wall-clock solve latency as a function of the number of parallel environments (equivalently, the number of simultaneously solved QP instances). \emph{TQPNet} exploits GPU-native batch matrix operations inherited from the \emph{MLP(ReLU)}-\emph{ADMM(ReLU)} architecture, maintaining near-constant per-batch latency as the environment count grows. In contrast, qpth~\cite{amos2017optnet} incurs significantly higher latency due to its reliance on interior-point solvers, and \emph{ADMM(ReLU)} with cold initialization requires substantially more refinement iterations without a learned warm start. 

Fig.~\ref{fig:quadruped_cost_and_violation} demonstrates (a) relative cost error against the OSQP optimum and (b) maximum constraint violation, as functions of refinement steps $K$. \emph{TQPNet} attains a cost error below $10^{-2}$ and a constraint violation below $10^{-2}$ within fewer than $K{=}20$ refinement steps, whereas cold-started \emph{ADMM(ReLU)} requires more than $K{=}60$ iterations to reach a comparable accuracy. This gap directly quantifies the value of the learned warm start: the \emph{MLP(ReLU)} initializer effectively compresses hundreds of ADMM iterations into a single forward pass while preserving certified solution quality. These results, combined with the batch-latency advantage in Fig.~\ref{fig:quadruped_batch}(b), show that \emph{TQPNet} delivers high-fidelity QP solutions at the throughput required by parallel reinforcement-learning pipelines in robotics.
\begin{figure}[htbp]
  \centering
  \begin{minipage}[t]{0.45\linewidth}
    \centering
    \includegraphics[width=\linewidth]{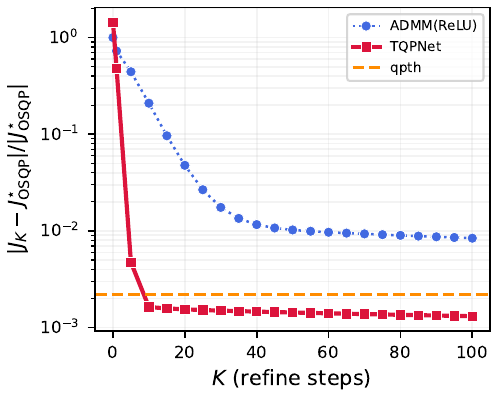}\\
  \end{minipage}
  \hfill
  \begin{minipage}[t]{0.45\linewidth}
    \centering
    \includegraphics[width=\linewidth]{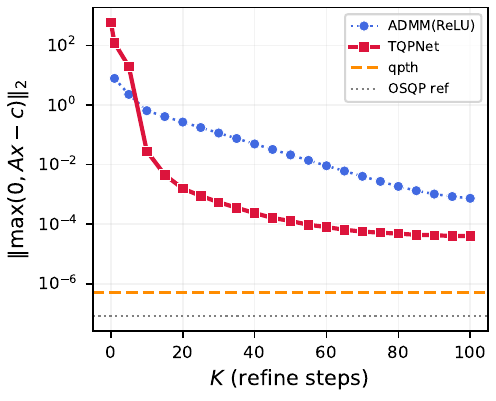}\\
  \end{minipage}
  \caption{Solution quality vs.\ refinement steps $K$ on the quadruped locomotion MPC. (a) Relative QP cost error vs.\ refinement steps $K$. (b) Constraint violation vs.\ refinement steps $K$.}
  \label{fig:quadruped_cost_and_violation}
  \vspace{-10pt}
\end{figure}
\vspace{-10pt}
\section{Conclusion}
\vspace{-10pt}
This paper proposed \emph{TQPNet}, a self-supervised neural compression framework for mpQP that compresses long ADMM optimization trajectories into a compact neural warm start followed by a few refinement iterations. By deriving a low-dimensional ADMM and its equivalent ReLU-layer representation, we use a hybrid \emph{MLP(ReLU)}-\emph{ADMM(ReLU)} architecture that combines amortized optimization with optimization-based refinement, enabling adaptive online certification. Experiments demonstrate substantial acceleration on large-scale mpQPs and real-time MPC applications while preserving strong solution accuracy and scalability. More broadly, this work suggests a new paradigm for neural compression of iterative optimization algorithms.

\textbf{Limitations} include the lack of a rigorous convergence analysis characterizing when the proposed low-dimensional ADMM outperforms the standard full-dimensional ADMM. In particular, since infinitely many COPE decompositions may exist for a given matrix, it remains unclear which decomposition yields the best convergence behavior.

Future work includes extending the proposed framework to more general classes of optimization, such as multiparametric mixed-integer QPs or nonlinear programs, as well as exploring hardware-aware implementations to further improve real-time performance.

\newpage
\begin{ack}
    This research was supported by the Ralph O'Connor Sustainable Energy Institute at Johns Hopkins University, and by the U.S. DOE Office of Science, ASCR program under the Scientific Discovery through Advanced Computing (SciDAC) Institute ``LEADS: LEarning-Accelerated Domain Science.'' The authors declare no competing interests.
\end{ack}

\bibliographystyle{unsrt}
\bibliography{refs}
\appendix
\section{Proofs}
\subsection{Proof of Lemma \ref{lemma_ADMM_ReLU_3}}\label{appendix_sec_lemma_ADMM_ReLU_3}
\begin{proof}
For the ADMM iteration \eqref{eqn_ADMM_bad}, putting $u$-update before $x$- and $z$- updates, which is equivalent to a special initialization for $u^0$ at the cost of one iteration lag in $u$-update and still preserves the convergence guarantee, results in
\begin{subequations}
 \begin{align}
    u^{k+1} &=u^{k} + z^{k}-Ax^{k}\label{eqn_u_k1}\\
    x^{k+1}&= H\left(-q(p)+\rho A^\top\left(z^k+u^{k+1}\right)  \right)\label{eqn_x_k1}\\
    z^{k+1} &= \max\left(c(p), Ax^{k+1}-u^{k+1} \right)\label{eqn_z_k1}
\end{align}   
\end{subequations}
Substituting Eqn.~\eqref{eqn_u_k1} into Eqn.~\eqref{eqn_x_k1}, and then substituting Eqns.~\eqref{eqn_u_k1}--\eqref{eqn_x_k1} into Eqn.~\eqref{eqn_z_k1}, results in the ReLU-like layer \eqref{eqn_ADMM_ReLU_like}, which completes the proof.
\end{proof}

\subsection{Proof of Lemma \ref{lemma_D_U_D_V}}\label{appendix_sec_lemma_D_U_D_V}
\begin{proof}
Based on the definition of $U,V$ in \eqref{eqn_cope}, we have
\[
    \begin{aligned}
         UU^\top &= (A_{i_1,j_1})^2e_{i_1}^m (e_1^o)^\top e_1^o (e_{i_1}^m)^\top + \cdots + (A_{i_o,j_o})^2e_{i_o}^m (e_o^o)^\top e_o^o(e_{i_o}^m)^\top\\ 
         &= \sum_{i_k=1}(A_{1,j_k})^2 e_1^m (e_1^m)^\top+\cdots + \sum_{i_k=m}(A_{m,j_k})^2 e_m^m (e_m^m)^\top\\
         &=\left[\begin{array}{cccc}
             A_{1,:}A_{1,:}^\top &  & & \\
             & A_{2,:}A_{2,:}^\top & & \\
             &  & \ddots & \\
             &  & & A_{m,:}A_{m,:}^\top
         \end{array}\right]\in\rr^{m\times m}
    \end{aligned}
\]
and 
\[
    \begin{aligned}
        VV^\top &= e_{j_1}^n (e_1^o)^\top e_1^o(e_{j_1}^n )^\top +  e_{j_2}^n (e_2^o)^\top e_2^o(e_{j_2}^n )^\top + \cdots + e_{j_o}^n (e_o^o)^\top  e_o^o(e_{j_o}^n)^\top\\
        &=e_{j_1}^n (e_{j_1}^n )^\top +e_{j_2}^n (e_{j_2}^n )^\top + \cdots + e_{j_o}^n (e_{j_o}^n)^\top\\
        &=\sum_{j_k=1} e_1^n(e_1^n)^\top  + \sum_{j_k=2} e_2^n(e_2^n)^\top  + \cdots + \sum_{j_k=n} e_n^n(e_n^n)^\top\\
        &=\left[\begin{array}{cccc}
             \mathrm{nnz}(A_{:,1}) & &  & \\
             & \mathrm{nnz}(A_{:,2}) &   & \\
              & &\ddots & \\
              & & & \mathrm{nnz}(A_{:,n})
        \end{array} \right]\in\rr^{n\times n},
    \end{aligned}
\]
To clarify the concept, we offer an example: given a $3\times 5$ matrix with $9$ non-zero elements as follows,
\[
A = \left[\begin{array}{ccccc}
    1 & 0 & 4 & 6 & 8 \\
    0 & 0 & 5 & 0 & 9 \\
    2 & 3 & 0 & 7 & 0
\end{array} \right],
\]
then its COPE decomposition is shown below:
\[
U = \left[\begin{array}{ccccccccc}
    1 & 0 & 0 & 4 & 0 & 6 & 0 & 8 & 0\\
    0 & 0 & 0 & 0 & 5 & 0 & 0 & 0 & 9\\
    0 & 2 & 3 & 0 & 0 & 0 & 7 & 0 & 0
\end{array} \right],~V = \left[\begin{array}{ccccccccc}
    1 & 1 & 0 & 0 & 0 & 0 & 0 & 0 & 0  \\
    0 & 0 & 1 & 0 & 0 & 0 & 0 & 0 & 0  \\
    0 & 0 & 0 & 1 & 1 & 0 & 0 & 0 & 0  \\
    0 & 0 & 0 & 0 & 0 & 1 & 1 & 0 & 0  \\
    0 & 0 & 0 & 0 & 0 & 0 & 0 & 1 & 1 
\end{array} \right]
\]
with
\[
D_U=UU^\top = \left[\begin{array}{ccc}
    117 & 0 & 0 \\
    0 & 106 & 0 \\
    0 & 0 & 62
\end{array}\right],~ D_V=VV^\top = \left[\begin{array}{ccccc}
     2 & 0 & 0 & 0 & 0 \\
     0 & 1 & 0 & 0 & 0 \\
     0 & 0 & 2 & 0 & 0 \\
     0 & 0 & 0 & 2 & 0 \\
     0 & 0 & 0 & 0 & 2 
\end{array}\right].
\]
This completes the proof.
\end{proof}

\subsection{Proof of Lemma \ref{lemma_closed_form}}\label{appendix_sec_lemma_closed_form}
\begin{proof}
The KKT condition of the optimization problem \eqref{eqn_ADMM_y} is shown below:
\[
\begin{aligned}
  &y - (V^\top x^k-u^k) - U^\top \beta=0 \\
  &s= Uy-c(p)\geq0\\
 &s_i \beta_i=0,~s_i\geq0,~\beta_i\geq0,~i=1,\cdots,m
\end{aligned}
\]
where $\beta\in\rr^m$ and $s\in\rr^{m}$ denote the Lagrangian and slack variables, respectively. Eliminating $y$ via $y=V^\top x^k-u^k+U^\top\beta$ can result in a trivial linear complementarity problem as follows,
\[
\begin{aligned}
  &s= U(V^\top x^k-u^k + U^\top \beta)-c(p)=Ax^k-Uu^k+D_U\beta-c(p)\geq0\\
  &s_i \beta_i=0,~s_i\geq0,~\beta_i\geq0,~i=1,\cdots,m
\end{aligned}
\]
which admits a closed-form solution $\beta=\max\left(0,D_U^{-1}(-Ax^k+Uu^k+c(p))\right)$, due to $D_U$ being a positive diagonal matrix. This completes the proof.
\end{proof}

\subsection{Proof of Theorem \ref{theorem_low_dim_ADMM_ReLU}}\label{appendix_sec_low_dim_ADMM_ReLU}
\begin{proof}
By Lemma \ref{lemma_closed_form}, we can replace \eqref{eqn_ADMM_y} with Eqn.~\eqref{eqn_y_update_closed_form} and eliminate $y^{k+1}$, which results in
\[
\begin{aligned}
    \beta^{k}&=\max\left(0,D_{U}^{-1}(-A x^k+Uu^k+c(p))\right)\\
    x^{k+1} &=(Q+\rho D_V)^{-1}\left(\rho V(V^\top x^k-u^k+U^\top\beta^{k}+u^k)-q(p) \right)\\
    &=(Q+\rho D_V)^{-1}\left( \rho D_Vx^k+\rho A^\top\beta^{k}-q(p)\right)\\
    u^{k+1} &= u^k +(V^\top x^k-u^k+U^\top\beta^{k})-V^\top x^{k+1} =V^\top (x^{k} - x^{k+1}) + U^\top \beta^k
\end{aligned}
\]
Furthermore, we have
\[
\begin{aligned}
    \beta^{k+1}&=\max\left(0,D_{U}^{-1}(-A x^{k+1}+Uu^{k+1}+c(p))\right)\\
    &\quad\Big(\text{where } Uu^{k+1}=UV^\top (x^{k} - x^{k+1}) + UU^\top \beta^k=A(x^{k} - x^{k+1}) + D_U \beta^k\Big)\\
    &=\max\left(0,D_{U}^{-1}(-A x^{k+1}+ A(x^{k} - x^{k+1}) + D_U \beta^k+c(p))\right)\\
    &=\max\left(0,\beta^k-D_{U}^{-1} A(2x^{k+1} - x^k) + D_{U}^{-1}c(p)\right)
\end{aligned}
\]
By adopting a special initialization $u^0=0$ and $\beta^0=\max(0,D_U^{-1}(c(p)-Ax^0))$, we can eliminate the $u$-update and obtain a low-dimensional ADMM iteration
\begin{equation}\label{eqn_lw_ADMM}
\begin{aligned}
x^{k+1}&=(Q+\rho D_V)^{-1}\left( \rho D_Vx^k+\rho A^\top\beta^{k}-q(p)\right)\\
\beta^{k+1}&=\max\left(0,\beta^k - D_{U}^{-1} A(2x^{k+1} - x^k) +D_U^{-1}c(p)\right)
\end{aligned} 
\end{equation}
which can be easily rewritten as a ReLU formulation as shown in Eqn. \eqref{eqn_ADMM_ReLU_2}. Moreover, Eqn. \eqref{eqn_ADMM_ReLU_2} involves the matrix-vector multiplication with the dimension $n+m$, thus the computation complexity is $O((n+m)^2)$. This completes the proof.
\end{proof}

\subsection{Proof of Theorem \ref{theorem_main}}\label{appendix_sec_theorem_main}
\begin{proof}
The Karush--Kuhn--Tucker (KKT) condition (see \cite[Ch.\ 5]{boyd2004convex}) for mpQP \eqref{eqn_QP} is given by
\begin{subequations}\label{eqn_KKT}
\begin{align}
Qx+q(p)-A^\top \beta &=0, \label{eqn_KKT_a}\\
y-Ax+c(p)&=0, \label{eqn_KKT_b}\\
y\geq 0,\quad \beta\geq 0,\quad y\pdot \beta&=0,
\label{eqn_KKT_c}
\end{align}
\end{subequations}
where $\pdot$ denotes the Hadamard product. Eqn. \eqref{eqn_KKT_c} can be equivalently rewritten as $\beta- \max(0,\beta-y)=0 \Leftrightarrow \beta- \max(0,\beta-Ax+c(p))=0$
where the equivalence follows by substituting Eqn. \eqref{eqn_KKT_b}. Then, combining this expression with \eqref{eqn_KKT_a} yields the KKT-informed residual defined in \eqref{eqn_r_definition}.
The piecewise affine structure of the KKT mapping in \eqref{eqn_KKT} implies, through standard Hoffman-type error bounds and metric subregularity results \cite{rockafellar1998variational}, that the distance to the solution set is locally upper bounded by the norm of the KKT residual, so there exists a constant $\kappa(p)>0$ such that $\mathrm{dist}\big((x,\beta),\mathcal{S}(p)\big)
\leq \kappa(p)\left\|r(x,\beta;p)\right\|_2,$ for all $(x,\beta)$ in a neighborhood of $\mathcal{S}(p)$, which completes the proof.
\end{proof}

\vspace{-10pt}
\section{Algorithms}
\vspace{-10pt}
\begin{algorithm}[H]
    \caption{Projected sampling strategy}\label{alg_projected_sampling}
    \textbf{Input}: the given domain $\mathcal{P}$ and the required number of data points $N$.
    \vspace*{.1cm}\hrule\vspace*{.1cm}
    1.\quad Construct an enlarged domain $\Tilde{\mathcal{P}}$ such that $\mathcal{P}\subseteq\Tilde{\mathcal{P}}$;\\
    2.\quad Sample the data points $\{p_i \}_{i=1}^N$ uniformly from the enlarged domain  $\Tilde{\mathcal{P}}$;\\
    3.\quad Project the sampled data points onto the original domain $\mathcal{P}$ by solving a projection problem: $  p_i^{\text{proj}} = \argmin_{p\in\mathcal{P}}\|p-p_i\|_2^2,~\forall i=1,\cdots,N$\\
    \textbf{Output}: $\mathcal{D}=\{p_i^{\text{proj}}\}_{i=1}^N$;
\end{algorithm}
\vspace{-10pt}
\begin{algorithm}[H]
    \caption{Unified offline training and online  inference }\label{alg_offline_online}
    \textbf{Offline training}. \\
    \textbf{Input}:  $\mathcal{D}$ collected by Algorithm \ref{alg_projected_sampling}, learning rate $\eta$, batch size $B$, number of epochs $N_{\mathrm{epoch}}$, offline stopping tolerance $\epsilon_{\text{off}}$.
    \vspace*{.1cm}\hrule\vspace*{.1cm}
    \textbf{for} $\mathrm{epoch}=1: N_{\mathrm{epoch}}$
    \begin{enumerate}[label*=\arabic*., ref=\theenumi{}]
    \item Split the dataset $\mathcal{D}$ into mini-batches $\mathcal B\subset \mathcal{D}$ with $|\mathcal B|=B$ and calculate the mini-batch loss $\mathcal{L}_{\mathcal B}(\theta)=\frac{1}{|\mathcal B|}\sum_{p\in\mathcal B}\log\big(\big\|r\big(\mathrm{ADMM\text{-}ReLU}_{K,\rho}(\mathcal{N}(p;\theta,M,L);p);p\big)\big\|_2^2+1\big)$;
    \item Update the neural network parameters using automatic differentiation and a gradient-based optimizer such as Adam or quasi-Newton (L-BFGS): $\theta \leftarrow \theta-\eta \nabla_\theta \mathcal{L}_{\mathcal B}(\theta)$;
    \item If $\mathcal{L}(\theta)\le \epsilon_{\mathrm{off}}$, break;
    \end{enumerate}
    \textbf{end}\\
\textbf{Output}: trained parameters $\theta$.
\vspace*{.1cm}\hrule\vspace*{.1cm}
\textbf{Online inference (supporting batch scenario) with triggered trustworthy correction}. \\
\textbf{Input}: trained parameters $\theta$,  matrix $P=\mathrm{col}(p_1,\cdots,p_{n_b})$ consisting of a batch of parameter vectors, online tolerance $\epsilon_{\mathrm{on}}$.
\vspace*{.1cm}\hrule\vspace*{.1cm}
    1.~~Compute the batch of neural proxies \textbf{in parallel}:
    {\scriptsize
    \[
    \begin{aligned}
    \left(\begin{bmatrix}
        x_1\\
        \beta_1
    \end{bmatrix},\cdots,\begin{bmatrix}
        x_{n_b}\\
        \beta_{n_b}
    \end{bmatrix}\right)=\left(\mathrm{ADMM\text{-}ReLU}_{K,\rho}\big(\mathcal{N}(p_1;\theta,M,L);p_1\big),\cdots,\mathrm{ADMM\text{-}ReLU}_{K,\rho}\big(\mathcal{N}(p_{n_b};\theta,M,L);p_{n_b}\big)\right)
    \end{aligned}
    \]
    }
    2.~~Compute the certificate value \textbf{in parallel}: 
    \[
    \bm{\varepsilon}=\left(\left\|r\big(x_1,\beta_1;p_1\big)\right\|_2^2,\cdots,\left\|r\big(x_{n_b},\beta_{n_b};p_{n_b}\big)\right\|_2^2\right)
    \]
3.~~If $\|\bm{\varepsilon}\|_{\infty}\leq\epsilon_{\mathrm{on}}$, accept the neural prediction and return;\\
4.~~Otherwise, find indices $\mathcal{I}$ such that $\bm{\varepsilon}_{\mathcal{I}}>\epsilon_{\mathrm{on}}$;\\
5.~~Update $(\{x\}_{\mathcal{I}},\{\beta\}_{\mathcal{I}})$ via low-dimensional ADMM-ReLU \eqref{eqn_ADMM_ReLU_2} in batch scenario;\\
\textbf{Output}: $(\{x\}_{1:n_b},\{\beta\}_{1:n_b})$.
\end{algorithm}

\section{Additional Ablation Results}\label{app:ablation}

\subsection{Cross-scale \texorpdfstring{$K_{\text{train}}$}{K\_train} scaling rule}\label{app:ablation_2d}

To verify that the unrolling depth must scale with problem size, we
replicate the ablation at the smaller scale $n{=}10$, $m{=}20$,
$n_{p}{=}4$. Table~\ref{tab:ablation_cross_scale} reports the median
test rel error (at fixed $K_{\mathrm{eval}}{=}10$) across both scales
and all six $K_{\text{train}}$ values, and
Fig.~\ref{fig:app_ablation_n10} shows the corresponding
training and inference curves at $n{=}10$.
The saturation point shifts from $K_{\text{train}}{=}10$ at $n{=}10$
(rel error $\sim 5\!\times\!10^{-3}$) to $K_{\text{train}}{=}25$ at
$n{=}100$ (rel error $\sim 5\!\times\!10^{-2}$); shallower unrolling at
$n{=}100$ remains in the unstable regime documented in the main text.
This is consistent with the contraction rate of ADMM: the operator's
contraction factor approaches $1$ as the conditioning of
$Q + \rho A^{\top}A$ grows with $n$, so a deeper unroll is needed for
$\mathrm{ADMM}_{K}(\mathrm{MLP}(p))$ to enter the neighborhood of
$x^{\star}$ where back-propagation reflects warm-start error.
Empirically, $K_{\text{train}} \approx n/4$ suffices.

\begin{table}[h!]
\centering
\caption{Median test-set rel error after fixed $K_{\mathrm{eval}}{=}10$ refinement at two scales. All entries trained for $2000$ epochs.}
\label{tab:ablation_cross_scale}
\setlength{\tabcolsep}{6pt}
\small
\begin{tabular}{rrrcccccc}
\toprule
$n$ & $m$ & $n_p$ & $K{=}0$ & $K{=}1$ & $K{=}5$ & $K{=}10$ & $K{=}25$ & $K{=}50$ \\
\midrule
10  & 20  & 4  & $1.0\!\times\!10^{-2}$ & $1.8\!\times\!10^{-1}$ & $1.4\!\times\!10^{-2}$ & $5.3\!\times\!10^{-3}$ & $3.6\!\times\!10^{-3}$ & $1.4\!\times\!10^{-2}$ \\
100 & 200 & 20 & $3.5\!\times\!10^{-1}$ & $6.3\!\times\!10^{-1}$ & $3.6\!\times\!10^{-1}$ & $5.8\!\times\!10^{-1}$ & $5.7\!\times\!10^{-2}$ & $5.6\!\times\!10^{-2}$ \\
\bottomrule
\end{tabular}
\end{table}

\begin{figure}[H]
\centering
\begin{minipage}[t]{0.46\linewidth}
\centering
\includegraphics[width=\linewidth]{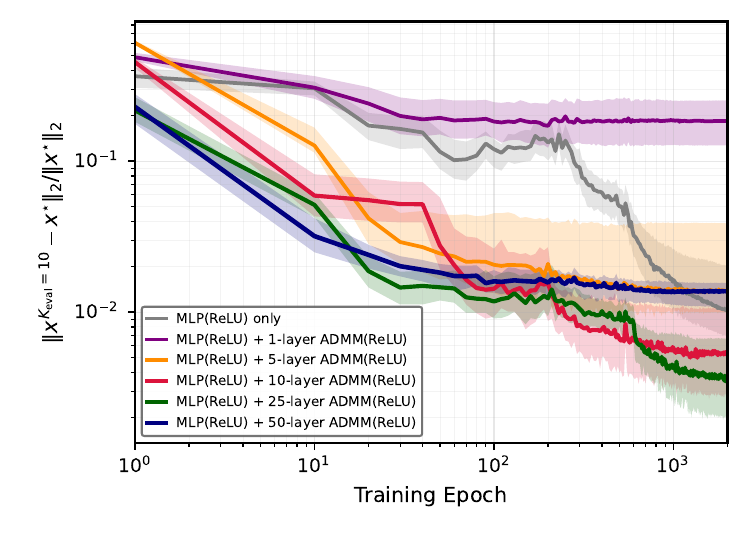}\\
{\small (a) Training.}
\end{minipage}
\hfill
\begin{minipage}[t]{0.46\linewidth}
\centering
\includegraphics[width=\linewidth]{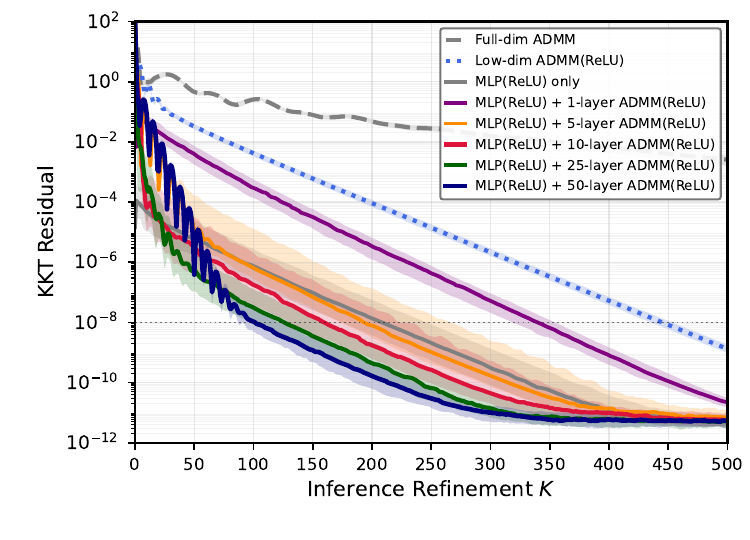}\\
{\small (b) Inference.}
\end{minipage}
\caption{Ablation of $K_{\text{train}}$ at $n{=}10$, $m{=}20$,
$n_{p}{=}4$.
(a) Test relative error vs.\ epoch at fixed $K_{\mathrm{eval}}{=}10$;
(b) post-training KKT residual vs.\ refinement steps $K$.
$K_{\text{train}}{=}10$ already saturates, in contrast to the phase transition in Fig.~\ref{fig:ablation}.}
\label{fig:app_ablation_n10}
\end{figure}

\subsection{Inference convergence detail}\label{app:ablation_full}

Table~\ref{tab:ablation} reports per-method iteration count $K^{\star}$
to KKT residual $\leq 10^{-7}$, mean wall-clock per QP, and the
implied speedup over cold-start ADMM(ReLU), corresponding to
Fig.~\ref{fig:ablation}b.
Fig.~\ref{fig:app_inference_time} re-plots the inference convergence
in cumulative wall-clock per QP at both scales:
$K_{\text{train}}{\geq}25$ reaches $\varepsilon{=}10^{-7}$ in roughly
$6$--$7\,\mathrm{ms}$ at $n{=}100$ versus $\sim 11\,\mathrm{ms}$ for
cold-start ADMM(ReLU), and the proportional speedup is even larger at
$n{=}10$ ($\sim 5\!\times$) where the MLP forward overhead becomes
negligible relative to the iteration savings.

\begin{table}[h]
\centering
\caption{Iterations $K^{\star}$ and wall-clock to KKT residual $\leq 10^{-7}$, median over $100$ test mpQPs at $n{=}100$, $m{=}200$, $n_p{=}20$. Speedup is relative to cold low-dim ADMM(ReLU); $\Delta K$ is the per-QP iteration reduction vs.\ that baseline. Full-dim ADMM does not reach the target within $K{=}10^{4}$.}
\label{tab:ablation}
\setlength{\tabcolsep}{4pt}
\small
\begin{tabular}{lcccc}
\toprule
Method & $K^{\star}$ & Time (ms) & Speedup & $\Delta K$ \\
\midrule
Full-dim ADMM & $>\!10^{4}$ & $>\!50.5$ & — & — \\
Low-dim ADMM(ReLU) & 4729 & 10.689 & 1.0$\times$ & — \\
MLP only & 3578 & 7.777 & 1.4$\times$ & 1132 \\
MLP + 1-layer ADMM(ReLU) & 3692 & 8.032 & 1.3$\times$ & 974 \\
MLP + 5-layer ADMM(ReLU) & 3728 & 7.751 & 1.4$\times$ & 1000 \\
MLP + 10-layer ADMM(ReLU) & 3392 & 7.305 & 1.5$\times$ & 1344 \\
MLP + 25-layer ADMM(ReLU) & 2800 & 6.460 & 1.7$\times$ & 1930 \\
MLP + 50-layer ADMM(ReLU) & 2778 & 6.034 & 1.8$\times$ & 1968 \\
\bottomrule
\end{tabular}
\end{table}

\begin{figure}[H]
\centering
\begin{minipage}[t]{0.46\linewidth}
\centering
\includegraphics[width=\linewidth]{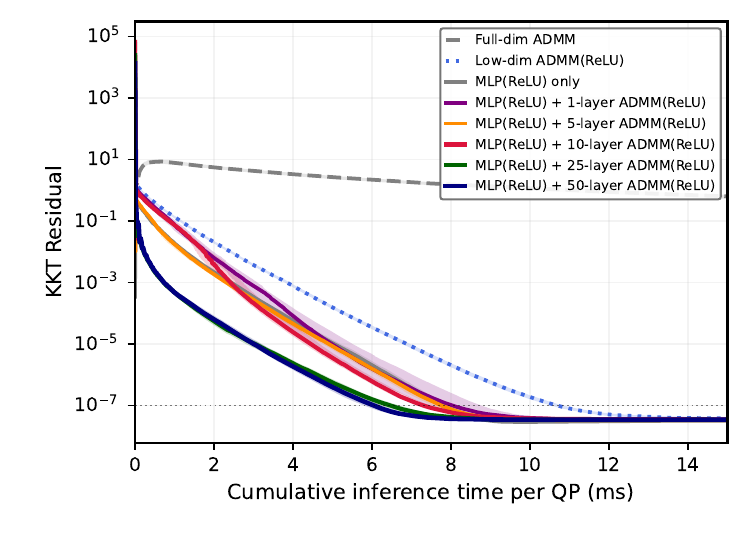}\\
{\small (a) $n{=}100$.}
\end{minipage}
\hfill
\begin{minipage}[t]{0.46\linewidth}
\centering
\includegraphics[width=\linewidth]{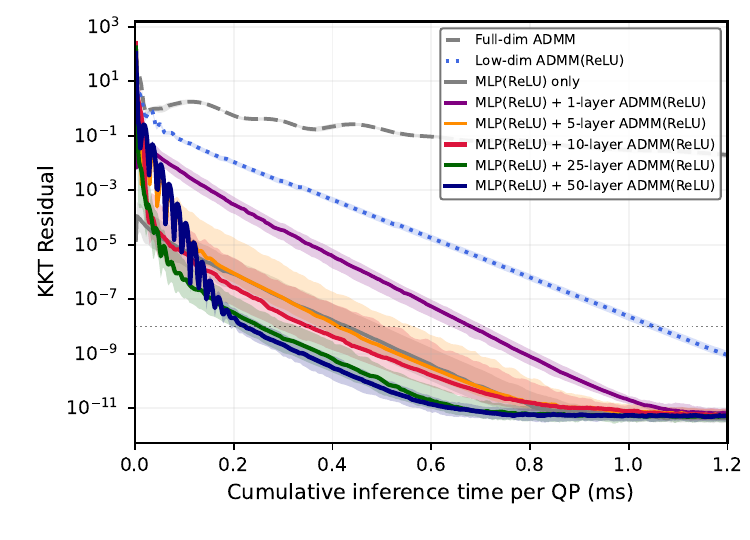}\\
{\small (b) $n{=}10$.}
\end{minipage}
\caption{Wall-clock view of inference convergence: KKT residual vs.\ cumulative time per QP. \emph{TQPNet}'s wall-clock advantage over \emph{ADMM(ReLU)} holds at both scales, with a larger proportional speedup at $n{=}10$.}
\label{fig:app_inference_time}
\end{figure}

\subsection{Implementation details}\label{app:ablation_impl}

\textbf{Architecture.} MLP(ReLU) backbone with depth $L{=}4$ and width
$M{=}512$ at $n{=}100$ ($M{=}256$ at $n{=}10$); the $x$-head is linear
and the $\beta$-head applies Softplus to enforce $\beta \geq 0$
($\sim$1.0M params at $n{=}100$, $\sim$0.2M at $n{=}10$).

\textbf{Unrolling depth.} $K_{\text{train}}{=}25$ for the out-of-distribution
study (Sec.~\ref{sec_numerical_experiments}), $K_{\text{train}}{=}10$ for the MPC
scalability and quadruped studies, and $K_{\text{train}}{=}30$ for the UAV experiment
(Appendix~\ref{app:appendix_uav}); the choice of $K_{\text{train}}$ is studied in
Fig.~\ref{fig:ablation} and Appendix~\ref{app:ablation_2d}.

\textbf{Optimization.} AdamW (lr $10^{-3}$, weight decay $10^{-5}$),
StepLR (step $300$, $\gamma{=}0.5$), gradient clipping at norm $10$,
batch size $256$, $2000$ epochs. Training uses
\texttt{torch.cholesky\_solve} (numerically stable backward); inference
uses pre-computed matmul.

\textbf{Problem distribution.} Sparse positive constraint matrix
$A \!\sim\! \mathrm{Uniform}([0, 3]^{m\times n})$ with column density
$0.3$; $c_{0} {=} A x_{\mathrm{ref}} {+} 1$ for a random reference
$x_{\mathrm{ref}}$ (active-set tightness controlled by the $+1$
offset); $p \!\sim\! \mathrm{Uniform}([{-}1, 1]^{n_p})$ for both
training ($N{=}6000$) and test ($N{=}100$).

\textbf{Reference solver.} OSQP at \texttt{eps\_abs}{$=$}\texttt{eps\_rel}{$=$}$10^{-10}$
with polishing for the test-set $x^{\star}$.

\textbf{Hardware.} Single NVIDIA RTX 3080, PyTorch 2.10, CUDA 12.8.

\newpage                                   
\section{Tables of Numerical Examples}                

\begin{table}[htbp]
\centering
\caption{Generalizability: median relative error ($n=100$, $m=200$, $n_p=20$). Train on $p \sim \mathcal{U}[-1,1]$.}
\label{tab:generalizability}
\begin{tabular}{lcccccc}
\toprule
Set & \multicolumn{2}{c}{$K=10$} & \multicolumn{2}{c}{$K=30$} & \multicolumn{2}{c}{$K=50$} \\
\cmidrule(lr){2-3}
\cmidrule(lr){4-5}
\cmidrule(lr){6-7}
 & TQPNet & LQPNet & TQPNet & LQPNet & TQPNet & LQPNet \\
\midrule
Train & 5.8e-02 & 1.5e-02 & 5.5e-02 & 1.4e-02 & 5.2e-02 & 1.3e-02 \\
Valid & 5.9e-02 & 1.6e-02 & 5.6e-02 & 1.5e-02 & 5.3e-02 & 1.4e-02 \\
Boundary & 2.2e-01 & 1.0e-01 & 2.1e-01 & 9.7e-02 & 2.0e-01 & 9.1e-02 \\
OOD & 3.5e-01 & 1.9e-01 & 3.4e-01 & 1.9e-01 & 3.2e-01 & 1.8e-01 \\
\bottomrule
\end{tabular}
\end{table}

\begin{table}[htbp]
\centering
\caption{Scalability on random MPC problems ($T{=}200$ closed-loop steps with shared process noise $w \sim \mathcal{U}[-0.3, 0.3]^{n_s}$). Online latency in ms, offline training time in s. TQPNet/LQPNet use a fixed network size: hidden width $M{=}512$ and depth $L{=}4$ across all scales.}
\label{tab:scalability_mpc}
\resizebox{\linewidth}{!}{
\begin{tabular}{rrc ccc ccc ccc cc c}
\toprule
 & & & \multicolumn{3}{c}{OSQP} & \multicolumn{3}{c}{TQPNet ($K{=}10$)} & \multicolumn{3}{c}{LQPNet ($K{=}10$)} & \multicolumn{2}{c}{Offline} & Explicit \\
\cmidrule(lr){4-6} \cmidrule(lr){7-9} \cmidrule(lr){10-12} \cmidrule(lr){13-14} \cmidrule(lr){15-15}
$n{+}m$ & $n_p$ & Speedup & RMSE & Cost & Latency & RMSE & Cost & Latency & RMSE & Cost & Latency & TQP & LQP & Latency \\
\midrule
  60 &  3 & $0.0\times$ & 0.221 & 33 & 0.03 & 0.221 & 33 & 1.31 & 0.221 & 33 & 1.23 & 65 & 24 & 0.57 \\
  140 &  4 & $0.1\times$ & 0.218 & 45 & 0.07 & 0.218 & 45 & 1.30 & 0.218 & 45 & 1.31 & 65 & 25 & 4.11 \\
  315 &  6 & $0.2\times$ & 0.217 & 63 & 0.21 & 0.217 & 63 & 1.23 & 0.217 & 63 & 1.20 & 65 & 35 & --- \\
  560 &  8 & $0.4\times$ & 0.263 & 124 & 0.59 & 0.264 & 124 & 1.31 & 0.263 & 124 & 1.33 & 71 & 43 & --- \\
  875 & 10 & $1.1\times$ & 0.217 & 99 & 1.38 & 0.217 & 99 & 1.27 & 0.217 & 99 & 1.23 & 84 & 65 & --- \\
  1050 & 12 & $1.6\times$ & 0.240 & 144 & 2.12 & 0.240 & 145 & 1.36 & 0.240 & 145 & 1.35 & 94 & 84 & --- \\
  1400 & 16 & $2.7\times$ & 0.198 & 130 & 3.55 & 0.198 & 130 & 1.33 & 0.198 & 130 & 1.37 & 118 & 112 & --- \\
  2100 & 24 & $5.0\times$ & 0.217 & 233 & 7.22 & 0.218 & 233 & 1.45 & 0.217 & 233 & 1.43 & 207 & 309 & --- \\
  2800 & 32 & $7.0\times$ & 0.225 & 332 & 13.06 & 0.225 & 334 & 1.88 & 0.225 & 333 & 1.49 & 334 & 590 & --- \\
\bottomrule
\end{tabular}}
\end{table}

\newpage

\section{Application to Quadrotor UAVs}\label{app:appendix_uav}

\textbf{TQPNet matches OSQP closed-loop tracking accuracy.}
We evaluate TQPNet on a quadrotor closed-loop tracking task following
\cite{bemporad2011decentralized}, with state $x\in\mathbb{R}^{13}$,
control $u\in\mathbb{R}^{4}$, and condensed-MPC dimensions $n{=}100$,
$m{=}300$ at horizon $N_{H}{=}25$ ($\Delta t{=}71\,\mathrm{ms}$).
The mpQP parameter $p\in\mathbb{R}^{17}$ comprises the current state
($13$ dims) and a four-component Fourier basis
$[\sin(\omega t), \cos(\omega t), \sin(2\omega t), \cos(2\omega t)]$
encoding the time-varying lemniscate reference. The closed loop is
simulated for $T{=}600$ steps with i.i.d.~process noise
($\sigma_{\mathrm{pos}}{=}3\,\mathrm{mm}$ shared across methods for
fair comparison). We compare TQPNet at $K \in \{10, 20, 30\}$
refinement steps against OSQP through its raw C API
(\texttt{prob.update}\,+\,\texttt{solve}, warm-started) at three
tolerance levels.

Table~\ref{tab:uav_mpc} reports tracking RMSE and per-step latency.
All TQPNet variants attain RMSE in $0.0297$--$0.0302\,\mathrm{m}$,
within $\sim 9\%$ of OSQP's $0.0273\,\mathrm{m}$ across the three
tolerance levels and improving over the cold-start ADMM(ReLU) baseline
($0.0306\,\mathrm{m}$); the closed-loop trajectories
(Fig.~\ref{fig:uav_appendix}) overlap visually because all methods
track the reference within the noise floor.
TQPNet additionally bounds worst-case execution at fixed $K$:
its per-step latency stays within $\sim 2.7\!\times$ of its median
(GPU launch jitter), whereas OSQP's iteration count varies up to
$1500$ at $\varepsilon{=}10^{-8}$ and its worst-case latency reaches
$10$--$25\!\times\!$ the median.
For hard real-time deployment where worst-case rather than mean
latency is binding, this matters.
The asymptotic batch advantage of the proposed architecture is realized
when many independent QPs must be solved in parallel, as demonstrated
by the quadrupedal locomotion experiment in Subsection~\ref{ssec:quadruped}.

\begin{table}[htbp]
\centering
\caption{Single-UAV lemniscate tracking ($N_H{=}25$, $n{=}100$, $m{=}300$, $T{=}600$). OSQP solved via raw C API with warm-started \texttt{prob.update} between steps. RMSE is comparable across methods. Latency and iteration counts report per-step median / maximum: OSQP exhibits worst-case latency $10$--$25\!\times\!$ its median, whereas TQPNet at fixed $K$ has bounded worst-case execution.}
\label{tab:uav_mpc}
\setlength{\tabcolsep}{6pt}
\small
\begin{tabular}{lccc}
\toprule
Method & RMSE (m) & Latency med/max (ms) & Iter med/max \\
\midrule
OSQP ($\varepsilon{=}10^{-6}$) & $0.0273$ & $0.73 \,/\, 7.03$  & $50 \,/\, 450$  \\
OSQP ($\varepsilon{=}10^{-7}$) & $0.0273$ & $0.75 \,/\, 13.48$ & $50 \,/\, 950$  \\
OSQP ($\varepsilon{=}10^{-8}$) & $0.0273$ & $0.75 \,/\, 21.10$ & $50 \,/\, 1500$ \\
\midrule
ADMM(ReLU) ($K{=}50$)          & $0.0306$ & $5.84 \,/\, 13.05$ & $50 \,/\, 50$   \\
TQPNet ($K{=}10$)              & $0.0302$ & $1.67 \,/\,  4.39$ & $10 \,/\, 10$   \\
TQPNet ($K{=}20$)              & $0.0301$ & $2.78 \,/\,  6.81$ & $20 \,/\, 20$   \\
TQPNet ($K{=}30$)              & $0.0297$ & $3.90 \,/\, 10.66$ & $30 \,/\, 30$   \\
\bottomrule
\end{tabular}
\end{table}

\begin{figure}[h!]
\centering
\begin{minipage}[t]{0.48\linewidth}
\centering
\includegraphics[width=\linewidth]{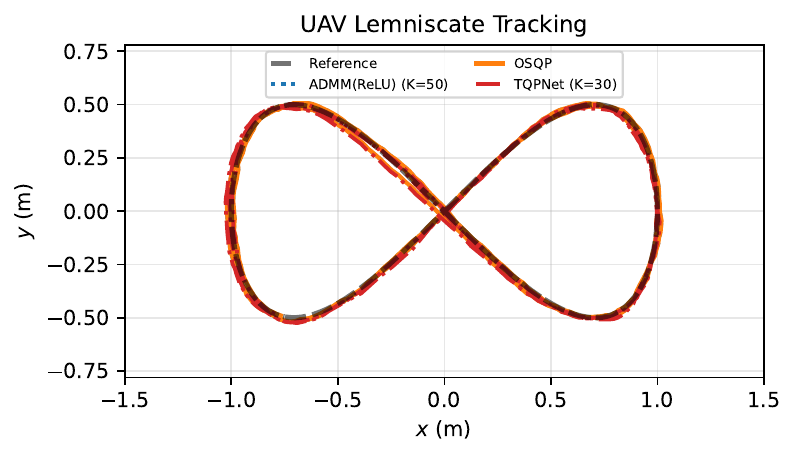}\\
{\small (a) Closed-loop trajectories.}
\end{minipage}
\hfill
\begin{minipage}[t]{0.48\linewidth}
\centering
\includegraphics[width=\linewidth]{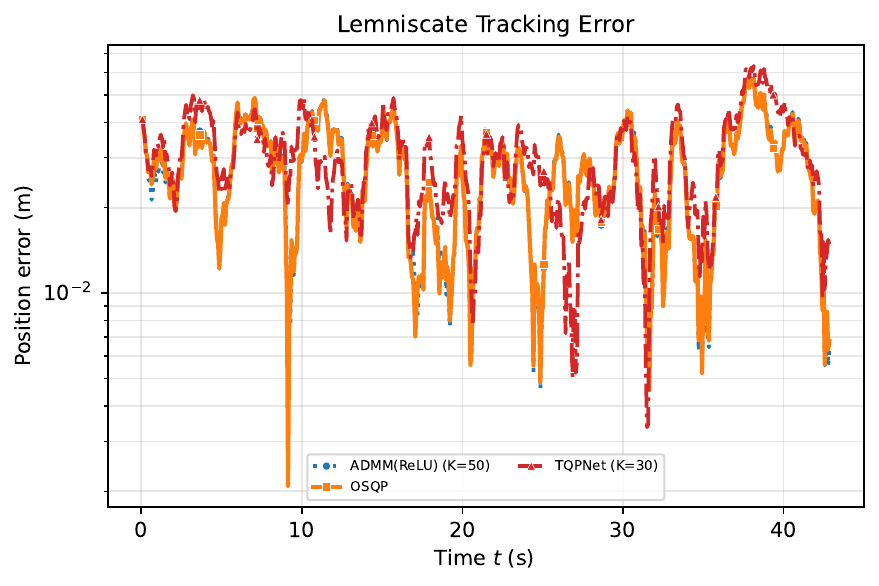}\\
{\small (b) Tracking error.}
\end{minipage}
\caption{Single-UAV lemniscate tracking ($T{=}600$), showing the best TQPNet variant ($K{=}30$) alongside cold-start ADMM(ReLU) and OSQP ($\varepsilon{=}10^{-8}$). All methods overlap at trajectory resolution and tracking errors are within the same magnitude; full method comparison in Table~\ref{tab:uav_mpc}.}
\label{fig:uav_appendix}
\end{figure}

\clearpage
% NeurIPS paper checklist omitted from the arXiv version.

\end{document}